\documentclass[preprint,12pt]{elsarticle}
\usepackage[margin=1in]{geometry}
\usepackage{times}
\usepackage{natbib}
\biboptions{round,sort&compress}
\usepackage{xcolor}
\usepackage{graphicx}
\usepackage{mathrsfs}
\usepackage{color}
\usepackage{amssymb}
\usepackage{amsfonts}
\usepackage{amsmath}
\usepackage{bbm}
\usepackage{nameref}
\usepackage{tikz}
\usepackage{stackengine}
\usepackage{lineno}
\usepackage{setspace}
\usepackage{algorithm}
\usepackage{algpseudocode}
\usepackage{subcaption}

\DeclareMathOperator*{\argmax}{argmax}
\usepackage{kantlipsum}
\allowdisplaybreaks
\usepackage{subcaption}
\usepackage{mwe}
\usepackage{diagbox}
\usepackage{multirow}
\usepackage{graphicx}
\usepackage{color, colortbl}
\definecolor{Gray}{gray}{0.9}
\usepackage{hyperref}
\usepackage{placeins}

\DeclareFontFamily{U}{rcjhbltx}{}
\DeclareFontShape{U}{rcjhbltx}{m}{n}{<->rcjhbltx}{}
\DeclareSymbolFont{hebrewletters}{U}{rcjhbltx}{m}{n}

\newcommand\barbelow[1]{\stackunder[1.2pt]{$#1$}{\rule{.8ex}{.075ex}}}

\usepackage{amsthm}
\newtheorem{prop}{Proposition}
\theoremstyle{definition}
\newtheorem{definition}{Definition}
\newtheorem*{definition*}{Definition}

\algtext*{EndFor}% Remove "end for" text
\algtext*{EndIf}% Remove "end if" text
\algtext*{EndFunction}% Remove "end function" text
\algtext*{EndWhile}% Remove "end while" text

\algrenewcommand\algorithmicforall{\textbf{foreach}}
\algrenewcommand\algorithmicindent{.8em}

\journal{Transportation Research Part C: Emerging Technologies}

\begin{document}

\begin{frontmatter}

%% Title, authors and addresses

\title{Dissolving the Segmentation of a Shared Mobility Market: \\A Framework and Four  Market Structure Designs}

%% use the tnoteref command within \title for footnotes;
%% use the tnotetext command for the associated footnote;
%% use the fnref command within \author or \address for footnotes;
%% use the fntext command for the associated footnote;
%% use the corref command within \author for corresponding author footnotes;
%% use the cortext command for the associated footnote;
%% use the ead command for the email address,
%% and the form \ead[url] for the home page:
%%
%% \title{Title\tnoteref{label1}}
%% \tnotetext[label1]{}
%% \author{Name\corref{cor1}\fnref{label2}}
%% \ead{email address}
%% \ead[url]{home page}
%% \fntext[label2]{}
%% \cortext[cor1]{}
%% \address{Address\fnref{label3}}
%% \fntext[label3]{}

%% use optional labels to link authors explicitly to addresses:
%% \author[label1,label2]{<author name>}
%% \address[label1]{<address>}
%% \address[label2]{<address>}

\author[mitcee]{Xiaotong Guo}
%\ead{xtguo@mit.edu}
\author[mitcee]{Ao Qu}
%\ead{qua@mit.edu}
\author[pkusg]{Hongmou Zhang\corref{cor}}
\ead{zhanghongmou@pku.edu.cn}
\author[mitdusp]{Peyman Noursalehi}
%\ead{peymano@mit.edu}
\author[mitdusp]{Jinhua Zhao}
%\ead{jinhua@mit.edu}

\address[mitcee]{Department of Civil and Environmental Engineering, Massachusetts Institute of Technology, Cambridge, MA 02139, United States}
\address[pkusg]{School of Government, Peking University, Beijing 100871, China}
\address[mitdusp]{Department of Urban Studies and Planning, Massachusetts Institute of Technology, Cambridge, MA 02139, United States}
\cortext[cor]{Corresponding author}

\begin{abstract}
In the governance of the shared mobility market of a city or of a metropolitan area, there are two conflicting principles: 1) the healthy competition between multiple platforms, such as between Uber and Lyft in the United States, and 2) economies of network scale, which leads to higher chances for trips to be matched, and thus higher operation efficiency, but which also implies monopoly. The current shared mobility markets, as observed in different cities in the world, are either monopolistic, or largely segmented by multiple platforms, the latter with significant efficiency loss. How to keep the competition between platforms, but to reduce the efficiency loss due to segmentation with new market designs is the focus of this paper. We first propose a theoretical framework of shared mobility market segmentation and then propose four market structure designs thereupon. The framework and four designs are first discussed as an abstract model, without losing generality, thus not constrained to any specific city. High-level perspectives and detailed mechanisms for each proposed market structure are both examined. Then, to assess the real-world performance of these market structure designs, we used a ride-sharing simulator with real-world ride-hailing trip data from New York City to simulate. The proposed market designs can reduce the total vehicle-miles traveled (VMT) by 6\% while serving more customers with 8.4\% fewer total number of trips. In the meantime, customers receive better services with on-average 5.4\% shorter waiting time. At the end of the paper, the feasibility of implementation for each proposed market structure is discussed.
\end{abstract}


\end{frontmatter}

%%
%% Start line numbering here if you want
%%
%\linenumbers

%% main text
\section{Introduction}

The Transportation Network Companies (TNCs), such as Uber and Lyft, are accountable for an additional 5.7 billion vehicle miles travelled (VMT) annually in just nine US cities~\cite{Schaller2018}. Although TNCs improve the convenience of travelers by allowing more travel options, their proliferation has in the meantime resulted in a marked increase in urban vehicle travel, leading to urban congestion, delays, and higher carbon emissions. While acknowledging the benefits of TNCs, such as enhanced convenience and accessibility of transportation services for urban residents, policymakers must also actively improve the efficiency of TNC markets, i.e., moving more people with fewer vehicles. This paper focuses on redesigning current shared mobility market structures to enhance market efficiency. The term "platforms" is used to refer to TNCs, as per the terminology in the platform economy literature~\cite{Thelen_2018}.
%{\color{red} It would be better to mention the ``efficiency,'' i.e., moving people with fewer cars, first, and then mention the ``inefficiency'' of it. It is a bit too abrupt here.}

The ``market segmentation'' in the ride-hailing market indicates that various competing platforms provide services to customers in the same area without coordination. This segmentation leads to lower customer and driver concentration on each platform, increasing the distances drivers must travel to pick up passengers. When a shared mobility platform handles more trips, it increases the likelihood of matching those trips with nearby drivers and pooling multiple trips together, resulting in more efficient use of vehicles and reduced congestion and emissions~\cite{Santi2014, Frechette2019}. Recent studies have analyzed the costs of market segmentation, also known as platform non-coordination~\cite{Sejourne_2018, ZhangHongmou2022EaDo, Kondor2022, wang2023quantifying}. \citet{Kondor2022} found that adding an additional ride-hailing platform could increase the number of vehicles by up to 67\% when platforms are not coordinated. The costs of market segmentation can be significant and ultimately lead to increased congestion and carbon emissions in cities. 

The segmentation of the market creates a crucial tension between two fundamental economic principles: fostering healthy market competition and achieving economies of scale. The challenge is finding a balance between promoting competition and encouraging the consolidation necessary to achieve these economies of scale. Instead of promoting a monopolistic market, the aim of this paper is to enhance the efficiency of the existing competitive but segmented market by encouraging \emph{cooperation} between platforms in a manner that decreases VMT and improves services for travelers. Cities around the globe have begun to recognize the advantages of cooperation among mobility service providers. For instance, the city of Z\"urich is investigating ways to improve collaboration among transport providers~\cite{Gansterer_Hartl_2018}. In the freight transportation sector, cooperation between multiple carriers is often achieved by reassigning transportation requests among them to reduce their total transportation costs or to maximize their profit from serving that requests, with potential savings ranging from 8\% to 60\%~\cite{Verdonck_2013, Sanchez_2016, Soysal_2018}. 
% Research has shown that collaboration among carriers can lower CO2 emissions by as much as 8-33\%\cite{Soysal_2018, Verdonck_2013}, with some estimates stating that savings could be as high as 60\%\cite{Sanchez_2016}. These results further emphasize the need to restructure the shared mobility market in such a way that it promotes and facilitates collaboration among platforms.

This paper focuses on resolving the fragmentation of shared mobility markets while maintaining healthy competition by proposing four possible market structures. The major contributions of this paper are:

\begin{itemize}
    \item Introduce a comprehensive theoretical framework for describing the shared mobility market structure and present four potential market structures to decrease the cost associated with market fragmentation.
    \item Describe two existing and four proposed shared mobility market structures from demand and supply information, service delivery, and payment flows perspectives.
    \item Design specific mechanisms for each proposed market structure.
    \item Evaluate the performance of each proposed market structure and its mechanism using real-world ride-hailing simulation based on New York City (NYC). The proposed market structures can reduce market friction by serving more customer requests with 6\% less total VMT, 8.4\% fewer total number of  trips, and 5.4\% reduced customer wait times.
\end{itemize}

It is worth noting that the fragmentation of the market in the ride-hailing industry can be partially alleviated through the actions of ``multi-homing'' drivers and customers who use multiple platforms. However, this passive form of collaboration only leads to marginal improvements in efficiency. As of May 2020, nearly 9\% of customers in the US were multi-homing customers~\cite{Uber-Lyft}. Although this helps to ``dissolve'' some of the segmentation between platforms, it cannot fundamentally change the market structure.

The structure of the article is as follows. In Section 2, we conduct a review of related literature on shared mobility markets and mechanism designs. In Section 3, we present a unified framework to describe shared mobility markets and a brief overview of possible market structures. Detailed market mechanisms are specified in Section 4. In Section 5, we present the results of numerical experiments conducted to evaluate the proposed market structures and mechanisms. Lastly, in Section 6, we conclude the paper by discussing the feasibility of each proposed market, identifying limitations, and outlining potential directions for future research.

\section{Literature Review} \label{sec:lit_review}

There is a vast body of research on shared mobility markets. This section will review the literature in three main areas: i) models and analyses of shared mobility markets, ii) shared mobility market structures and mechanisms, iii) and general market structures and mechanisms.

\subsection{Models and Analyses in Shared Mobility Markets}

\citet{Wang_Yang_2019} proposed a general framework and conducted a comprehensive literature review on shared mobility markets. There are two main perspectives of research in shared mobility markets: i) understanding demand and supply, and ii) platform operations. 

Understanding and predicting demand patterns are critical to platforms' operations. Studies have found that ride-hailing users are young, wealthy, and located in higher-density urban areas~\cite{Dias2017, Young2019}. \citet{Lavieri2019} described a comprehensive analysis of ride-hailing travel behavior, leading to policy implications about pooling acceptance and relationships with other modes.
On the other hand, state-of-the-art deep learning frameworks are applied to predict short-term passenger demand in the shared mobility system~\cite{Ke2017, Ke2019}.

Meanwhile, supply-side information and drivers' behaviors need to be understood to better operate shared mobility systems. 
\citet{Hall2018} utilized the survey and administrative data to provide a comprehensive analysis of the labor market with Uber drivers. \citet{Xu2020} discussed the supply curve of the ride-hailing system considering finite matching radius, which is built upon the ``Wild Goose Chase'' effect identified by \citet{Castillo2022}.

With the understanding of demand and supply in the shared mobility market, platforms need to design operation strategies to better serve customers and utilize driver supply. 
The operation strategies typically include dynamic pricing~\cite{LEI201977, Ban2021, Maljkovic2022, Banerjee2022, LIU2023103960}, customer-driver matching~\cite{Alonso-Mora2017, Bertsimas2019, DAGANZO2019213, BONGIOVANNI2022102835}, and vehicle rebalancing~\cite{Wen2017, Tsao2019, GUO2021, Guo2022}. 

\subsection{Shared Mobility Market Structures and Mechanisms}

Most studies have focused on the competition between two-sided platforms in the shared mobility market as pointed out by \citet{Wang_Yang_2019}. \citet{Zha_Yin_Yang_2016} analyzed the shared mobility market with an aggregated model focusing on customer-driver matching. Competition between multiple ride-sourcing platforms is investigated and their results suggested that the regulatory agency should encourage the merger of competing platforms if the matching friction between customers and drivers was sufficiently large. \citet{Cohen_Zhang} discussed the competition between two-sided platforms with a Multinomial Logit (MNL) behavior choice model and illustrated the existence and uniqueness of equilibrium of the competing market. They further proved that the ``coopetition,'' which is synonymic to the concept of cooperation in our paper, leads to a win-win situation where platforms, drivers, and customers benefit from the partnerships.

While cooperation between competing platforms brings additional benefits to the shared mobility system, few papers have explored \emph{competition with collaboration} in shared mobility markets under different market conditions and regulations. \citet{Shaheen_Cohen_2020} highlighted the business models and partnerships as one of the core enablers for facilitating the MOD system. However, 
how to redesign the shared mobility market structure to facilitate collaboration between competing platforms, establish feasible market mechanisms to allow such collaboration, and impose appropriate regulations still remains an open question.

\subsection{General Market Structures and Mechanisms}

Outside the shared mobility field, market structures that enable collaboration between competitors have been studied.
\citet{Verdonck_2013} conducted a survey on horizontal cooperation in logistics, leading to improvements in companies' productivity and level of service. Horizontal cooperation refers to sharing customer orders (demand side) or vehicle capacities (supply side). Multiple operation-level techniques are introduced to facilitate horizontal cooperation, including auction-based mechanisms and bilateral swapping. In forestry transportation, \citet{FRISK2010} proposed a collaboration mechanism for cost allocation and demonstrated that the proposed mechanism leads to an additional 9\% saving to each company while better planning strategies could only save around 5\%. \citet{Kotzab_Teller_2003} focused the coopetition in the European grocery industry. They showed that collaboration and competition can be performed simultaneously even under a competition intense scenario, and the coopetition offered companies with better management solutions and promoted economies of scale.

In this paper, our goal is to improve the efficiency of shared mobility markets by utilizing theories and experiences from other fields to design market structures that promote collaboration among shared mobility platforms. To the best of the authors' knowledge, this is the first paper to analyze the shared mobility market from a systematic perspective and offer potential market structures to dissolve the market segmentation. In the following sections, we will present a comprehensive theoretical framework for understanding the shared mobility market, as well as four potential market structures and mechanisms that take into account various levels of collaboration between platforms.

\section{Market Representation}

In this section, we will present a unified framework for describing shared mobility market structures and present market representation from information collection, service delivery, and payment flow perspectives. A summary of proposed market structures is presented at the end of the section as well.

\subsection{Preliminary}

For the operations of ride-hailing platforms, a methodology framework for operating dynamic shared mobility platforms with high-capacity vehicles is proposed by \citet{Alonso-Mora2017}. Given a set of requests $\mathcal{R}$ and available vehicles $\mathcal{V}$ in a decision time interval of length $\Delta$, a pairwise shareability network, namely RV-graph, is constructed, indicating the possibility for any two requests to share the same trip and for any vehicle to pick up any request. Using the RV-graph as a baseline, a set of candidate trips $\mathcal{T}$ is enumerated and an RTV-graph is formulated.
An Integer Linear Program (ILP) is then solved based on the RTV-graph to generate the optimal request-trip-vehicle assignment. 
Based on the optimal assignment, vehicle routes were generated and the vacant vehicles are repositioned. 

Let $\boldsymbol{m}$ indicate the optimal assignment between requests, trips, and vehicles. The ILP for computing the optimal assignment is denoted as $\boldsymbol{f}(\boldsymbol{\Sigma}, G)$, where $\boldsymbol{\Sigma}$ stands for the objective of the assignment problem and $G = (V, E)$ is the RTV-graph. Therefore, the proposed method for solving the assignment problem can be represented as $\boldsymbol{m} = \boldsymbol{f}(\boldsymbol{\Sigma}, G)$. The objective $\boldsymbol{\Sigma}$ can be generalized to any possible objectives related to vehicle-request and request-request matchings. For instance, \citet{Alonso-Mora2017} found an assignment that minimized the sum of delays over all assigned requests and penalties for unsatisfied requests.

Additionally, we introduce the following basic definition from the graph theory to help construct the unified theoretical framework~\cite{graph_theory}.

\begin{definition}[Subgraph]
Given a graph $G=(V, E)$ with vertex set $V$ and edge set $E$, a \emph{subgraph} $G' = (V', E')$ of a graph $G$ is a graph $G'$ whose vertex set and edge set are subsets of those of $G$, i.e., $V' \subseteq V$, $E' \subseteq E$.
\end{definition}

\subsection{A Unified Framework}

The essence of a shared mobility market is a set of rules to conduct an assignment between drivers and customers. Therefore, there are three layers of abstraction in this formulation process: 1) the underlying supply and demand, 2) the actual assignment between them, or what we call ``matching'' in all prior papers~\cite{Santi_Resta_Szell_Sobolevsky_Strogatz_Ratti_2014, Vazifeh_Santi_Resta_Strogatz_Ratti_2018}, and 3) the ``market structure'' or ``market mechanism,'' which rules how this assignment is going to be after supply and demand are realized. Typically, this rule set is driven by an objective, e.g., profit maximization, customers' travel time minimization, or VMT minimization, and is constrained by a list of factors, e.g., market segmentation, and how the segmentation could be ``dissolved'' to different extents.

More formally, let's denote a bipartite set $\mathcal{Q}:= \mathcal{R} \bigcup \mathcal{V}$ in a shared mobility market, where $\mathcal{R}$ represents customers (who request trips from origins to destinations) and $\mathcal{V}$ corresponds to drivers.
Given the underlying feasibility constraints for a ``matching'' between customers and drivers, and between customers\footnote{The typical feasibility constraints include the maximum wait time, the maximum delay time, and vehicle capacity constraints in the dynamic ride-hailing system.}, a set of candidate trips $\mathcal{T}$ and the corresponding RTV-graph $G_\mathcal{Q} = (E_\mathcal{Q}, V_\mathcal{Q})$ can be constructed. 
After specifying certain objective $\boldsymbol{\Sigma}$, the optimal assignment $\boldsymbol{m}$ can be found by solving the ILP based on the RTV-graph $G_\mathcal{Q}$, i.e., $\boldsymbol{m} = \boldsymbol{f}(\boldsymbol{\Sigma}, G_\mathcal{Q})$.

In addition to the underlying supply, demand and feasibility constraints, a market structure $\mu$ imposes extra constraints over the RTV-graph $G_\mathcal{Q}$. 
A specific market structure $\mu$ leads to a modified RTV-graph $G_\mathcal{Q}^\mu = (E_\mathcal{Q}^\mu, V_\mathcal{Q}^\mu)$ which is a \emph{subgraph} of the RTV-graph $G_{\mathcal{Q}}$, i.e., $E_\mathcal{Q}^\mu \subseteq E_\mathcal{Q}$, $V_\mathcal{Q}^\mu \subseteq V_\mathcal{Q}$.
Since the feasibility constraint for constructing the original RTV-graph is independent of market structures, considering market structures leads to subgraphs of the RTV-graph $G_{\mathcal{Q}}$.
For example, in a segmented market, the driver set $\mathcal{V}$ is partitioned into several non-overlapping subsets $\mathcal{V} = \mathcal{V}_1\bigcup\cdots \bigcup\mathcal{V}_n$, and \emph{vice versa} for customers. 
Only trips that include customers from the same subset $\mathcal{R}_i, \forall i,$ remain in the candidate trip set $\mathcal{T}$, hence it can be partitioned into several non-overlapping subsets $\mathcal{T} = \mathcal{T}_1\bigcup\cdots \bigcup\mathcal{T}_n$.
Edges between trips and drivers are retained only if they share the same index (under the same platform).
The segmented market leads to a modified RTV-graph $G_\mathcal{Q}^\mu$ described above and it can be further used to solve the optimal assignment problem, i.e., calculating $\boldsymbol{m} = \boldsymbol{f}(\boldsymbol{\Sigma}, G^\mu_\mathcal{Q})$.

\subsection{Basic Assumptions}

In microeconomics, a market is defined as a collection of buyers and sellers that, through their actual or potential interactions, determine the price of a product or set of products~\cite{Pindyck_Rubinfeld_2013}.
A shared mobility market can be treated as two separate markets linked by platforms: a market with drivers being service sellers and platforms being service buyers; a market with customers being service buyers and platforms being service sellers.
Let $\mathcal{P}$ denote the set of platforms and $|\mathcal{P}| > 1$ represents a segmented shared mobility market.
In this paper, without the loss of generality, we consider the competing market as the two-platform scenario, i.e., $\mathcal{Q} = (\mathcal{R}_1 \bigcup \mathcal{V}_1) \cup (\mathcal{R}_2 \bigcup \mathcal{V}_2)$.
Expanding the formulation to include multiple competing platforms in the market is a straightforward process.
A single-platform market is indicated by $\mathcal{Q} = \mathcal{R} \bigcup \mathcal{V}$. Assume all customers $\mathcal{R}$ need to be served by one platform in the platform set $\mathcal{P}$. First, we make the following assumptions and definitions for the shared mobility market studied in this paper:
\begin{enumerate}
    \item The shared mobility market is an open market where customers could leave the market if the price is unacceptable.
    \item A homogeneous fleet of vehicles with a capacity of 4 operated by a set of drivers. 
    \item Each driver contracts with at most one platform and each customer requests from only one platform (no multi-homing customers and drivers).
    \item The pricing scheme for a platform $i \in \mathcal{P}$ can be represented by $(\boldsymbol{p}_i, \boldsymbol{q}_i, \boldsymbol{o}_i)$. 
    Platform $i \in \mathcal{P}$ charges a customer with an OD-pair $(s,t)$ the price $\boldsymbol{p}_{i}(d_{st}, \tau_{st})$ for a dedicated trip and $\boldsymbol{q}_{i}(d_{st}, \tau_{st})$ for a shared trip, where $d_{st}$ and $\tau_{st}$ are the shortest path distance and travel time (calculate with constant vehicle speed $\Bar{v}$) between the origin $s$ and the destination $t$, respectively.
    For the driver who serves a trip that consists of either a single customer or multiple customers, the platform $i \in \mathcal{P}$ pays $\boldsymbol{o}_i(\hat{d}, \hat{\tau})$ to the driver, where $\hat{d}$ and $\hat{\tau}$ represents distance and time for both non-occupied (pick-up) and occupied trip, respectively.
\end{enumerate}

Under the assumption of pricing scheme $(\boldsymbol{p}, \boldsymbol{q}, \boldsymbol{o})$ and constant vehicle speed $\Bar{v}$, the optimal assignment $\boldsymbol{m}$ with minimum VMT provides each platform with the largest revenue. For all markets throughout the paper, the objective function $\Sigma$ equipped with market structures $\mu$ is the VMT minimization, and $C(\boldsymbol{m})$ denotes the overall VMT for an optimal assignment $\boldsymbol{m} = \boldsymbol{f}(\boldsymbol{\Sigma}, G^\mu_Q)$.

\subsection{Status Quo Market}

In this subsection, we describe two existing \emph{status quo} shared mobility markets. 
While the shared mobility platforms provide massive convenience to travelers, limited interventions and regulations are imposed by governmental authorities~\cite{Edelman_Geradin_2015}.
There are two types of \emph{status quo} markets: single-platform market and multi-platform market.
A typical single-platform market is the Chinese ride-hailing market, where DiDi served 93\% of total daily active users in 2019 \cite{Didi}. 
As for the multi-platform market, the American market with Uber and Lyft competing for the market share is an emblematic one.
The latest data shows that Uber served 71\% of the market share nationwide while Lyft served the remaining 29\% for April 2020 \cite{Uber-Lyft}. 
In certain cities and neighborhoods, the gap may be even smaller (Detroit has a nearly 50/50 market share between Uber and Lyft).

There are three components in the shared mobility market: Driver, Customer, and Platform. To describe the structure of each shared mobility market, we introduce four flows between these components:
\begin{itemize}
    \item \textbf{\color{red} Demand information flow}: Flow of customer request information.
    \item \textbf{\color{blue} Supply information flow}: Flow of driver location and occupancy information.
    \item \textbf{\color{orange} Payment flow}: Flow of money paid by customers to use shared mobility services. 
    \item \textbf{Physical service delivery flow}: Flow of physical service delivery from drivers to customers. 
\end{itemize}

Physical service delivery flow is trivial to discuss since it always directly runs from drivers to customers in any shared mobility market.
Thus, in the following discussions, we only focus on {\color{red}demand information}, {\color{blue}supply information} and {\color{orange}payment flow}, which are denoted by {\color{red}red}, {\color{blue}blue}, and {\color{orange}orange} arrows, respectively, in the following figures.
Figure \ref{fig:status_quo_market} illustrates the \emph{status quo} markets, which are straightforward.
In both single-platform and multi-platform \emph{status quo} markets, platforms collect demand information from customers and supply information from drivers, and allocate drivers to customers. 
The payment flow goes from customers to drivers via platforms.

\begin{figure}[!h]
\begin{subfigure}{.5\textwidth}
  \centering
  \includegraphics[width=.95\linewidth]{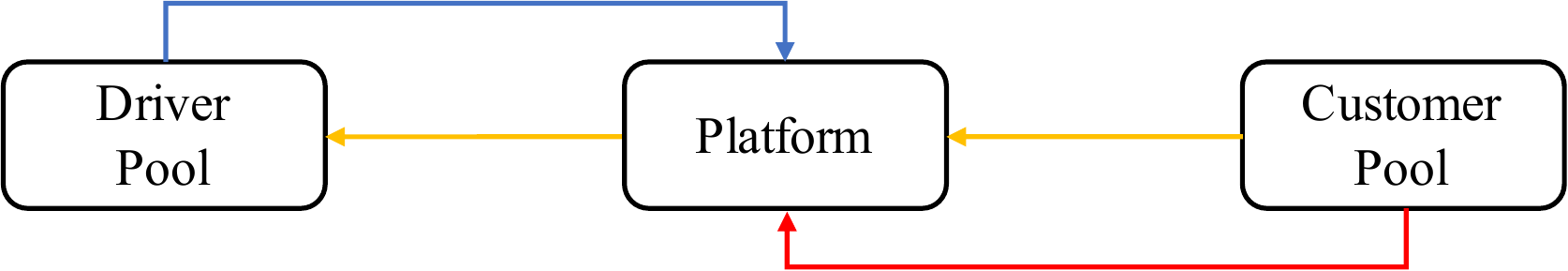}  
  \caption{single-platform market}
  \label{fig:status_quo_market_1}
\end{subfigure}
\begin{subfigure}{.5\textwidth}
  \centering
  \includegraphics[width=.95\linewidth]{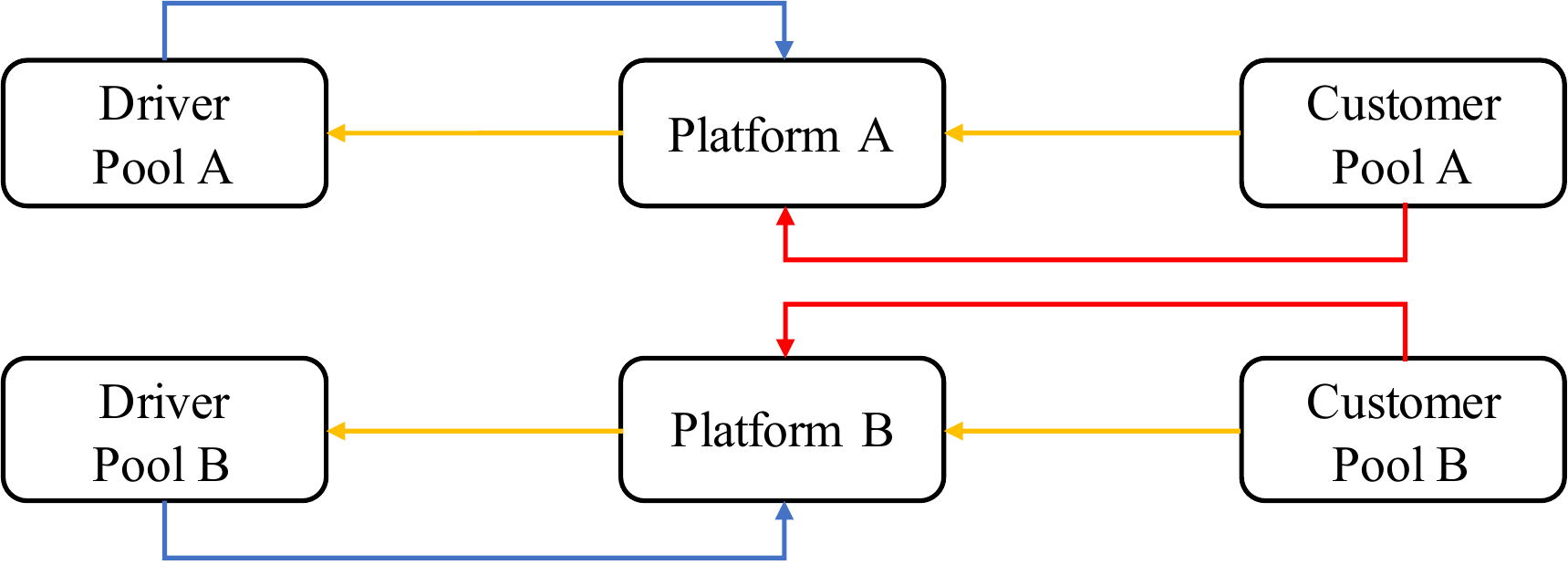}  
  \caption{multi-platform market}
  \label{fig:status_quo_market_2}
\end{subfigure}
\caption{\emph{Status quo} shared mobility markets.}
\label{fig:status_quo_market}
\end{figure}

In the \emph{status quo} single-platform market, market structure $\mu$ put no additional constraints over $\mathcal{Q}$.
For the \emph{status quo} multi-platform market with $n$ platforms, market structure $\mu$ divides driver set into $\mathcal{V} = \mathcal{V}_1 \bigcup \cdots \bigcup \mathcal{V}_n$ and customer set into $\mathcal{R} = \mathcal{R}_1 \bigcup \cdots \bigcup \mathcal{R}_n$.

The optimal assignment $\boldsymbol{m}$ only includes edges within subsets of $\mathcal{V}_i \bigcup \mathcal{R}_i, \forall i =1,...,n$. Assuming a set of demand and supply $\mathcal{Q}$ in a shared mobility market with a set of platforms $\mathcal{P}$ offering both dedicated and ride-pooling services, each platform $i \in \mathcal{P}$ tends to maximize its profit under a pricing scheme $(\boldsymbol{p}_i, \boldsymbol{q}_i, \boldsymbol{o}_i)$.
Considering all possible markets under this assumption but with different market structures, the \emph{status quo} single-platform and multi-platform markets serve as two extreme cases regarding the system efficiency (or the overall VMT).

To gain the largest revenue in a single-platform market, the platform tries to find the optimal driver-customer and customer-customer matchings within $\mathcal{Q}$ to minimize the overall VMT while serving all customers.
For any fragmented market, where the set of customers and drivers $\mathcal{Q}$ are divided into multiple disjoint subsets, each platform $i \in \mathcal{P}$ solves the optimal matching problem with its own set of customers and drivers $\mathcal{R}_i \bigcup \mathcal{V}_i$. 
Market fragmentation leads to VMT losses compared to single-platform markets. Although the monopoly in a market is typically considered as the source of inefficiency, we analyze the market efficiency purely from the VMT perspective, and a monopoly (single-platform) market leads to the minimum VMT among all possible shared mobility markets.

A single-platform market induces an unfair and unhealthy market due to a lack of competition~\cite{FLEISHMAN_2019}.
Facing the situation that the \emph{status quo} multi-platform market yields the worst system efficiency, we proposed four possible markets with different market structures and mechanisms.
In each of the proposed market structures, the hard boundaries between segmented platforms are partially ``dissolved.''
This dissolution of platform boundaries reduces the constraints of cross-platform trip matching, thus enabling more sharing opportunities and further reducing the overall VMT.

\subsection{Bilateral Trading Market}

The first proposed market, \emph{bilateral trading} market, improves the system efficiency of the \emph{status quo} multi-platform market by allowing trading, in an encrypted way to protect the data security, of customer or driver information between platforms.
\textit{Bilateral trading} market offers platforms with choices for trading supply or demand information that they can not efficiently serve\footnote{Drivers or customers that give platforms low or negative revenues.}.
Supply or demand information is traded between any two platforms if both platforms can improve their revenues, which also reduces the overall VMT.
For example, in a market with two platforms $A$ and $B$, a customer requests a ride with platform $A$ and all available drivers from platform $A$ are far away from this customer. There is an available driver from platform $B$ who is close to this customer. 
By allowing bilateral trading in the market, platform $A$ could trade the customer request information to platform $B$ at an appropriate price such that both platforms and the customer gain benefits from the trading.
Figure \ref{fig:bilateral_trading_market} illustrates the bilateral trading market with Figure \ref{fig:bilateral_trading_market_1} showing the information flow and Figure \ref{fig:bilateral_trading_market_2} indicating the payment flow.
Based on relationships in \emph{status quo} multi-platform market, all three types of flow can move between platforms in the bilateral trading market.

\begin{figure}[!h]
\begin{subfigure}{.5\textwidth}
  \centering
  \includegraphics[width=.95\linewidth]{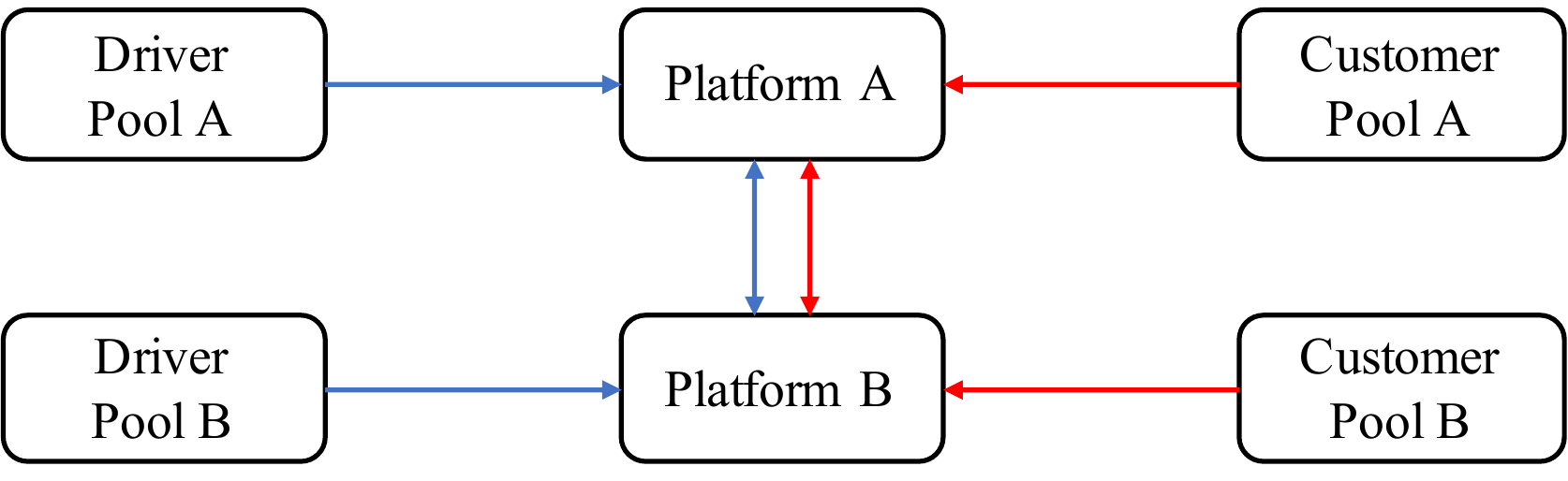}  
  \caption{Demand \& supply information flow}
  \label{fig:bilateral_trading_market_1}
\end{subfigure}
\begin{subfigure}{.5\textwidth}
  \centering
  \includegraphics[width=.95\linewidth]{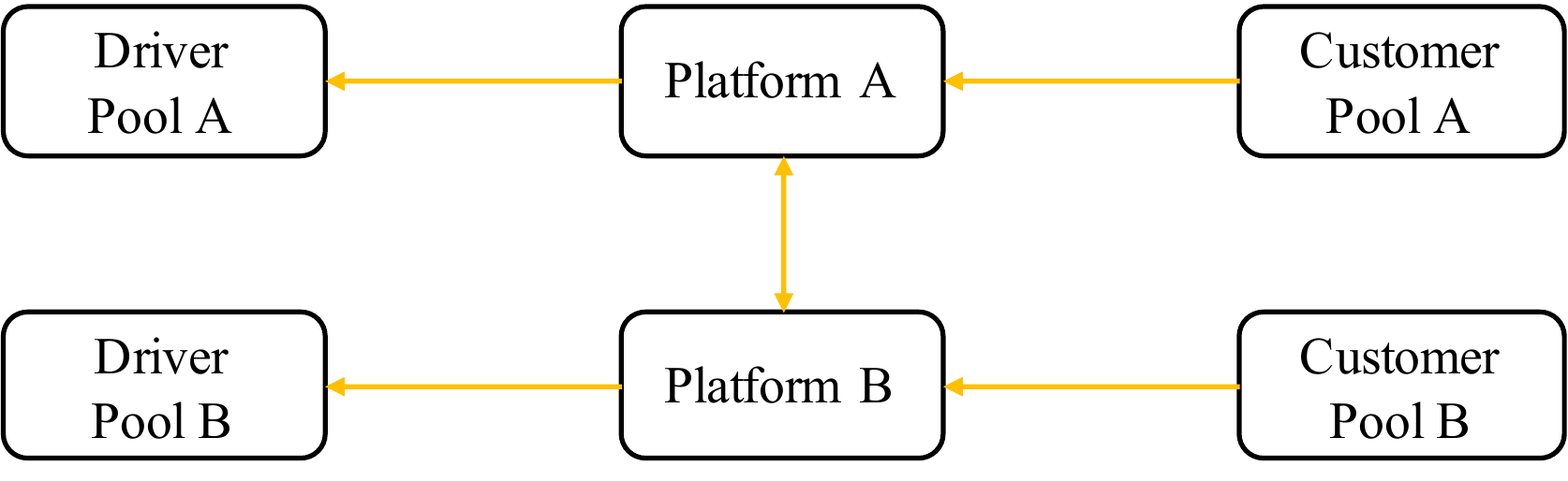}  
  \caption{Payment flow}
  \label{fig:bilateral_trading_market_2}
\end{subfigure}
\caption{Bilateral trading market illustration.}
\label{fig:bilateral_trading_market}
\end{figure}

% bilateral trading in unified framework
For the bilateral trading market, the market structure $\mu$ allows matchings $m_{ij}$ between $R_i$ and $V_j, \forall i,j = 1,...,n$ in the optimal assignment.
In theory, bilateral trading markets can be as efficient as single-platform markets with an infinite number of tradings.
However, only a limited number of platform pairs $(i,j)$ will trade in practice and feasible matchings $m_{ij}$ between $R_i$ and $V_j$ can be infeasible if trading information does not bring extra benefits to both platforms even though the system efficiency can be improved.

With three existing components, drivers, customers, and platforms, in the shared mobility market, possibilities for deriving different market structures are tightly restricted.
In the following three proposed markets, we introduce a new component, the central broker, which represents non-profitable governmental authorities, U.S. Department of Transportation (DOT) for instance, or non-governmental organizations to facilitate the cooperation between platforms.

\subsection{Central Trading Market}

\textit{Central trading} market is generalized from the bilateral trading market by introducing a central broker to conduct the supply or demand information trading between multiple platforms. 
Instead of trading bilaterally, multiple platforms can trade simultaneously with the help of a central broker.
Figure \ref{fig:central_trading_market} explains the central trading market. 
Demand information, supply information, and payment flow are moving between platforms through the central broker.
The central trading market is the most common type of market structure in reality, the stock market for instance.
The stock exchange has a similar role as the central broker in the shared mobility market, which offers a platform for stock trading between issuing companies and investors.

\begin{figure}[!h]
\begin{subfigure}{.5\textwidth}
  \centering
  \includegraphics[width=.95\linewidth]{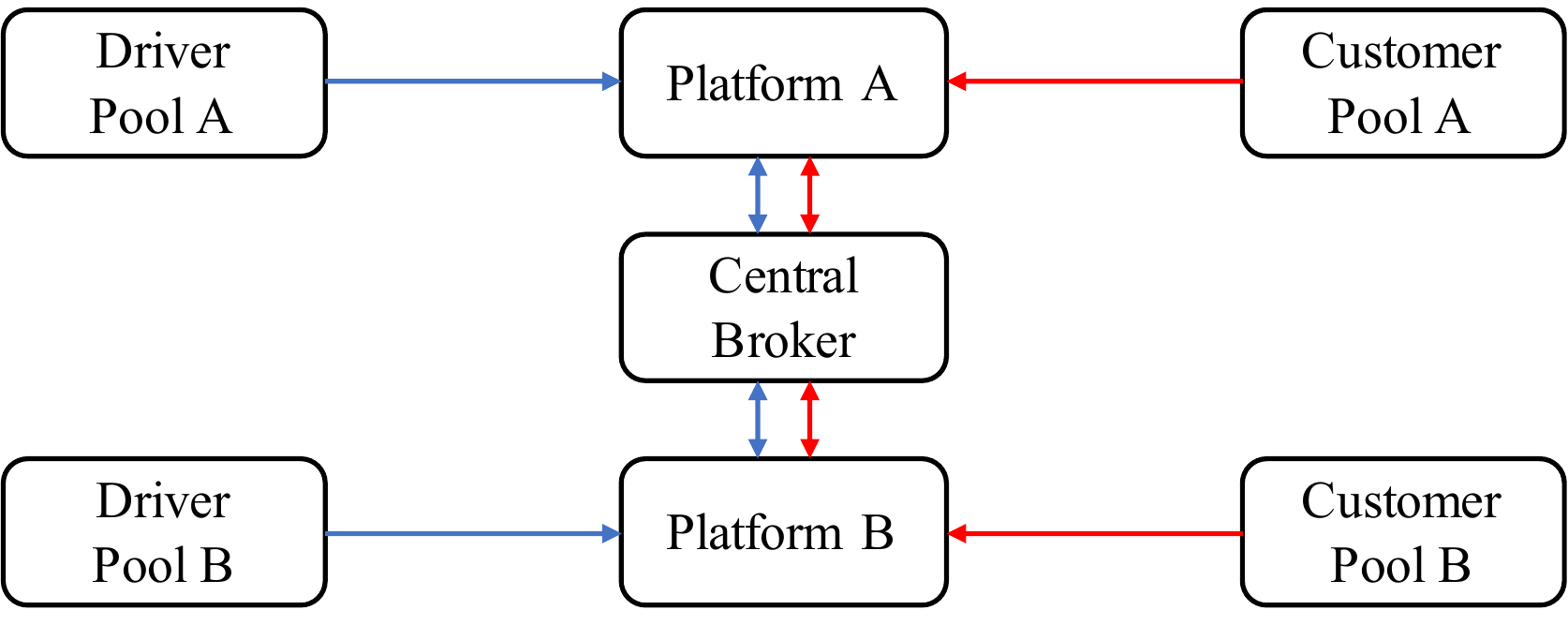}  
  \caption{Demand \& supply information flow}
  \label{fig:central_trading_market_1}
\end{subfigure}
\begin{subfigure}{.5\textwidth}
  \centering
  \includegraphics[width=.95\linewidth]{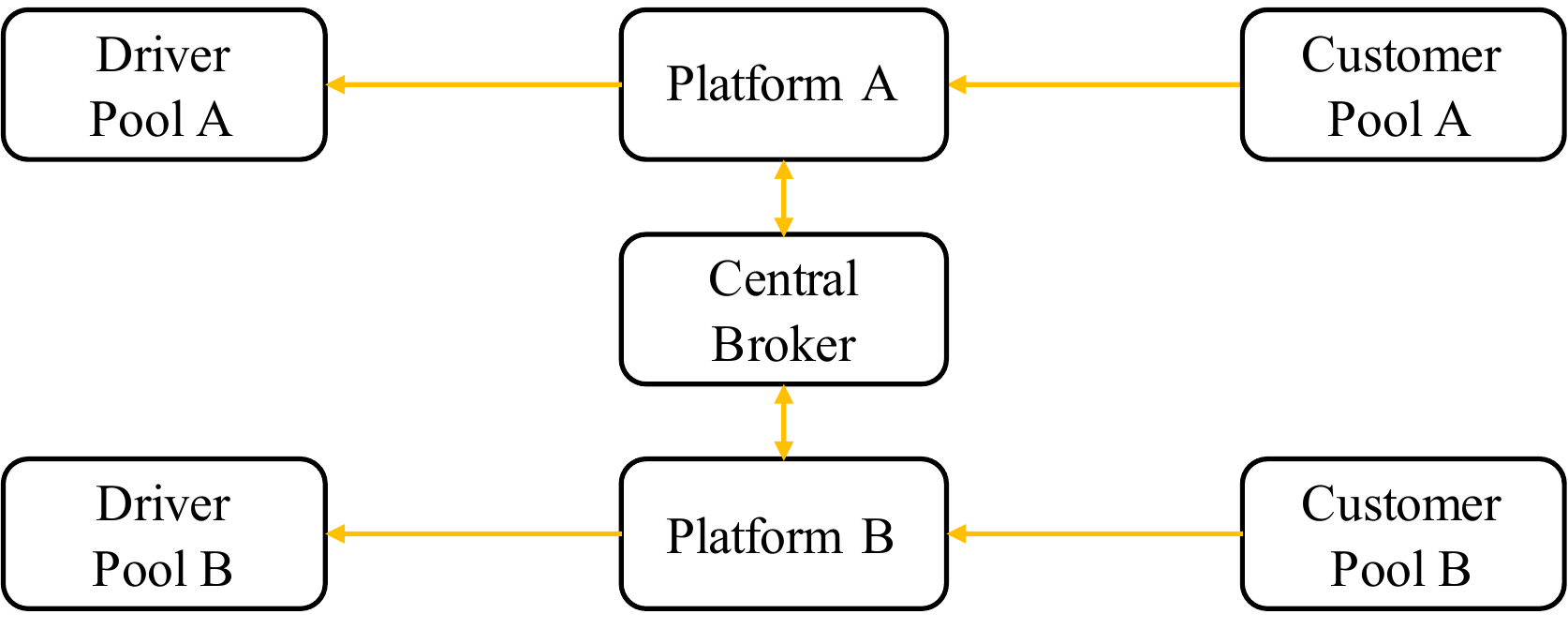}  
  \caption{Payment flow}
  \label{fig:central_trading_market_2}
\end{subfigure}
\caption{Central trading market illustration.}
\label{fig:central_trading_market}
\end{figure}

Similar to the bilateral trading market, the market structure $\mu$ of the central trading market permits feasible matchings $m_{ij}$ between $R_i$ and $V_j, \forall i, j = 1,...,n$. 
The feasibility of matchings depends on both spatiotemporal constraints and whether information tradings are beneficial for platforms. In a shared mobility market with multiple platforms, the central trading market gains extra benefits regarding the system efficiency compared to the bilateral trading market by introducing more trading opportunities.

\subsection{Cooperative Market}

\textit{Cooperative market} is a market where multiple platforms form an alliance and contributes their driver and customer information to a common pool. Platforms make an agreement on a common pricing scheme and profit distribution mechanism for the alliance. A central broker assigns drivers to customers in the common pool and distributes profit to platforms. This market can be as efficient as the \emph{status quo} single-platform market when all platforms in the market form a ``grand'' platform.

Figure \ref{fig:cooperative_market} illustrates the cooperative market. Figure \ref{fig:cooperative_market_1} shows that the central broker collects information from drivers and customers via platforms. The payment flow is displayed in Figure \ref{fig:cooperative_market_2}, where the central broker receives payments from customers and distributes them to platforms and then to drivers. For the cooperative market, the market structure $\mu$ allows feasible matchings $m_{ij}$ between $R_i$ and $V_j, \forall i, j \in \Bar{N}$, where $\Bar{N}$ indicates the set of platforms in the alliance.

\begin{figure}[!h]
\begin{subfigure}{.5\textwidth}
  \centering
  \includegraphics[width=.95\linewidth]{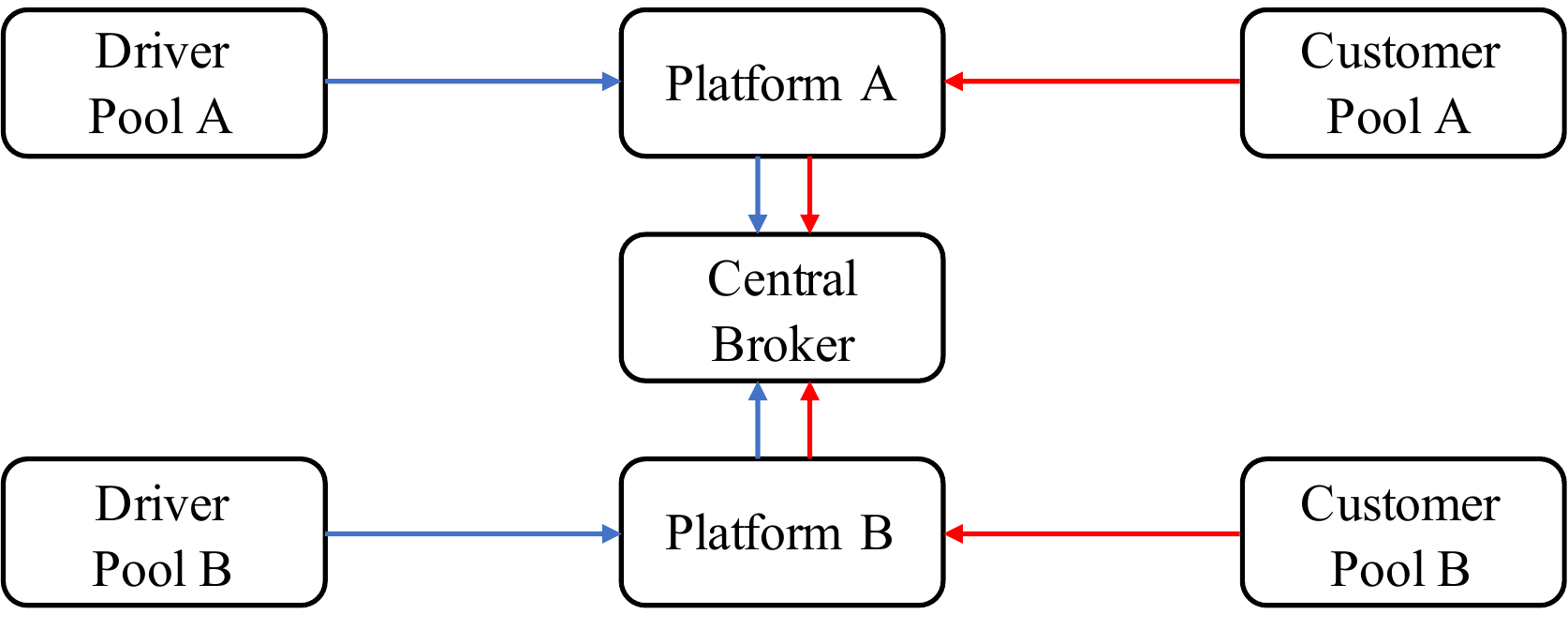}  
  \caption{Demand \& supply information flow}
  \label{fig:cooperative_market_1}
\end{subfigure}
\begin{subfigure}{.5\textwidth}
  \centering
  \includegraphics[width=.95\linewidth]{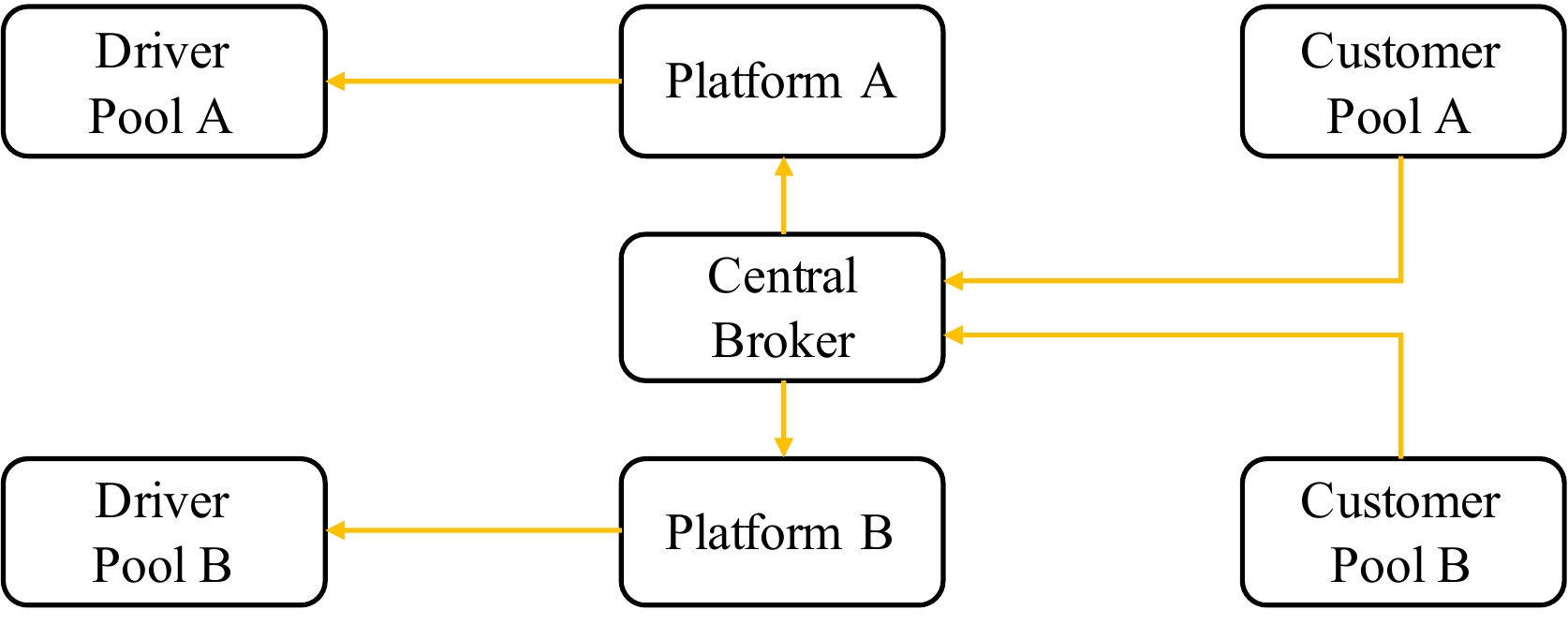}  
  \caption{Payment flow}
  \label{fig:cooperative_market_2}
\end{subfigure}
\caption{Cooperative market illustration.}
\label{fig:cooperative_market}
\end{figure}

\subsection{Shared Mobility Marketplace}

The central broker could also play a more fundamental role which gathers demand or supply information in the shared mobility market.
For the next proposed market, we assume that the central broker gathers demand information\footnote{The case where a central broker collects the supply information can be treated as an equivalent case.} and platforms gather supply information.

\emph{Shared mobility marketplace} is a market where the central broker acts as an auctioneer and sells demand information to platforms based on certain mechanisms, such as the single-item VCG (Vickrey-Clarke-Groves) mechanism. Given the location of their available drivers, platforms bid for customer requests, and the central broker distributes requests to platforms and charges the price of information. Platforms assign drivers to customers after getting demand information via the auction. The detailed mechanisms will be discussed in the following section.

Figure \ref{fig:shared_mobility_marketplace} explains the shared mobility marketplace. As shown in Figure \ref{fig:shared_mobility_marketplace_1}, platforms collect supply information from drivers and receive demand information from the central broker, who gathers demand information from customers.
Figure \ref{fig:shared_mobility_marketplace_2} describes the payment flow, where platforms receive payments from customers and a proportion of their revenue is used to pay drivers' salaries and another proportion to pay the price of information charged by the central broker.

\begin{figure}[!h]
\begin{subfigure}{.5\textwidth}
  \centering
  \includegraphics[width=.95\linewidth]{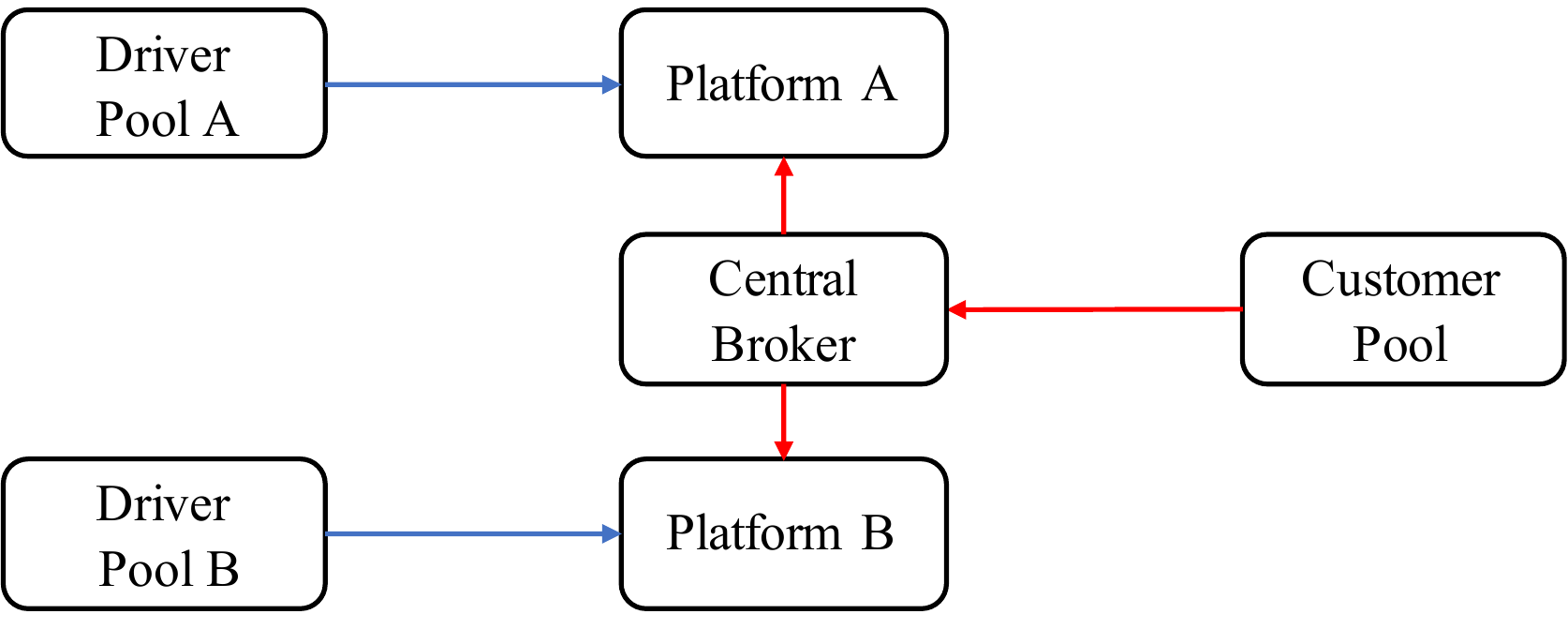}
  \caption{Demand \& supply information flow}
  \label{fig:shared_mobility_marketplace_1}
\end{subfigure}
\begin{subfigure}{.5\textwidth}
  \centering
  \includegraphics[width=.95\linewidth]{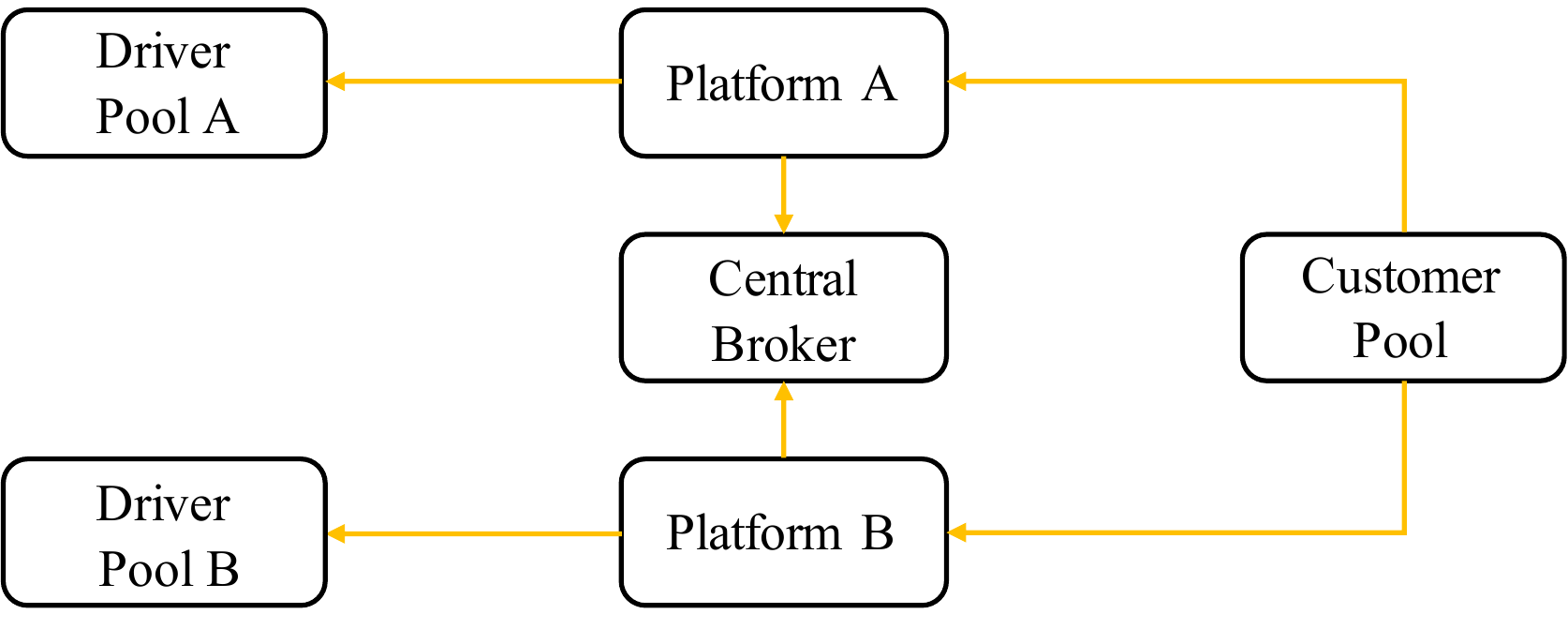}
  \caption{Payment flow}
  \label{fig:shared_mobility_marketplace_2}
\end{subfigure}
\caption{Shared mobility marketplace illustration.}
\label{fig:shared_mobility_marketplace}
\end{figure}

For the shared mobility marketplace, the driver set is split by $n$ platforms, i.e., $\mathcal{V} = \mathcal{V}_1 \bigcup \cdots \bigcup \mathcal{V}_n$. 
The market structure $\mu$ enables feasible matchings between $\mathcal{R}$ and $\mathcal{V}_i, \forall i=1,...,n$.
When platforms bid truthfully, indicating that platforms submit bids based on their true revenue for serving customers, the optimal assignment $\boldsymbol{f}(\cdot)$ is \emph{bona fide} VMT minimization.

\subsection{Summary}

Table \ref{tab:combined_table} summarizes key elements of the two existing and four proposed market structures and enumerates the roles of platforms and central brokers in each market structure.
There are two major factors to distinguish between different market structures: market segmentation and mediation of a central broker. Furthermore, the roles of platforms and central brokers are diverse across all markets. For \emph{status quo} single-platform and multi-platform markets, platforms are in charge of contracting drivers, collecting customer requests, matching, and setting pricing. For each proposed market, we allow a central broker to take over tasks from platforms.

\begin{table}[h!]
\centering
\begin{tabular}{| l | c | c |}
 \hline 
 \textbf{Market} & \textbf{Segmentation} & \textbf{Central broker} \\ [0.5 ex] 
 \hline
 I. \emph{Status quo} single-platform  & $\times$ & $\times$  \\
 II. Cooperative & \checkmark & \checkmark  \\
 III. Shared mobility marketplace & \checkmark & \checkmark  \\ 
 IV. Central trading & \checkmark & \checkmark  \\
 V. Bilateral trading & \checkmark & $\times$  \\
 VI. \emph{Status quo} multi-platform & \checkmark & $\times$  \\ [0.5 ex] 
 \hline
 \hline
 \textbf{{}Role of platform/central broker} & \textbf{Platform} & \textbf{Central broker} \\ [0.5 ex] 
 \hline
 1. Contracts with drivers  & I, II, III, IV, V, VI  & - \\
 2. Receives customer requests & I, II, IV, V, VI & III \\
 3. Matches customers with drivers & I, III, IV, V, VI & II  \\ 
 4. Supply/Demand information buyer & III, IV, V &  - \\
 5. Supply/Demand information seller & V & III, IV  \\
 6. Pricing & I, II,  IV, V, VI &  III \\
 7. Profit distribution & - & II \\
 \hline
\end{tabular}
\caption{Summary of different market structures. }
\label{tab:combined_table}
\end{table}

\section{Market Mechanisms}

In this section, we will explore the potential for market implementation by examining the specific mechanisms that can be put in place for each proposed market. This will be done by analyzing the underlying structure of each market and delving into the specifics of what would be required to make it a reality.

\subsection{Cooperative Market Mechanism}

The cooperative market is a market in which platforms form a multilateral alliance while each platform remains its own user base and fleet. Platforms contract with drivers and receive customers' request information while the central broker assigns customers to drivers. Platforms make an agreement on the pricing structure $(\boldsymbol{p}^*, \boldsymbol{q}^*, \boldsymbol{o}^*)$ and a central broker is responsible for assigning drivers to customers to maximize the overall profit of the alliance. In order to incentivize platforms to cooperate and participate in the alliance, a fair profit allocation mechanism has to be established.
This paper introduces three profit allocation mechanisms for the cooperative market. 

\subsubsection{Basic Definitions}

The profit allocation problem has been studied in the cooperative game theory, where players (platforms in the shared mobility market setting) are able to form binding commitments or coalitions.

First, we introduce the fundamental model in the cooperative game theory. Let $N = \{1,...,n\}$ be a finite set of players, each non-empty subset of $N$ is called a \textit{coalition} and $N$ is referred to as the \textit{grand coalition}. 
For each coalition $S$, the collected payoff value is defined as $v(S)$, where $v: 2^{N} \longrightarrow \mathbb{R}$ is a characteristic function associated with every coalition $S$ to a value $v(S)$. The pair $(N,v)$ is called a \textit{cooperative game}.

In a cooperative game $(N,v)$, the main focus is to find acceptable allocations of payoffs in the grand coalition $N$. 
Let's define an \textit{allocation} to be a collection of numbers $(x_1, x_2, ..., x_n)$ where $x_i$ denotes the value received by player $i$. An allocation $(x_1, x_2, ..., x_n)$ is \textit{efficient} if  $\sum_{i=1}^{n} x_i = v(N)$ and is \textit{individually rational} (IR) if $x_i \geq v(\{i\}), \forall i \in N$. Individual rationality implies that each player must receive at least as much value as that player receives without interacting with other players. Then, we introduce the key definition in the cooperative game \cite{Gillies1959}:

\begin{definition}[Core of cooperative games]
\label{def:core}
    The core of a cooperative game $(N, v)$ is a set of payoff allocations $\boldsymbol{x} = (x_1,x_2,...,x_n) \in \mathbb{R}^N$ satisfying:
    \begin{itemize}
        \item \textit{(Efficiency)} $\sum_{i \in N} x_i = v(N)$,
        \item \textit{(Coalitional Rationality)} $\sum_{i \in S} x_i \geq v(S), \forall S \subset N.$
    \end{itemize}
\end{definition}

The coalitional rationality property indicates that no coalition can upset allocations for the grand coalition in the core.  
In other words, no players can gain extra benefits by leaving the grand coalition. If such an allocation $(x_1, x_2, ..., x_n)$ exists, we can get a profit allocation mechanism that is stable and beneficial for every player in the game. However, a core does not guarantee to exist in general cooperative games. When evaluating a profit allocation mechanism, besides stability and being beneficial for platforms, fairness is also an important dimension to consider. A fair allocation indicates an equitable profit distribution between platforms without favoritism.
In the following subsection, we propose three profit allocation mechanisms based on allocation existence, fairness, and stability.

\subsubsection{Three Profit Allocation Mechanisms}

In the shared mobility cooperative market context, $N$ represents the set of platforms and $v$ denotes the net profit function under the pricing scheme $(\boldsymbol{p}^*, \boldsymbol{q}^*, \boldsymbol{o}^*)$. For a coalition $S$ with $m$ platforms, the allocation mechanisms generate a collection of numbers $(x_1,...,x_m)$ where $x_i$ denotes the profit received by platform $i$. A coalition is stable if it exists in the core of the cooperative game~\cite{Gillies1959}. In this section, we introduce three profit allocation mechanisms: Shapely value, Equal Profit Method (EPM), and contribution-based allocation mechanism.

The Shapley value \cite{shapley1988} is a wide-accepted unique allocation mechanism in the cooperative game. It ensures the existence and uniqueness of the profit allocation. However, this allocation mechanism does not consider fairness and is not guaranteed to be stable. For each player $i$ in a cooperative game $(N, v)$, the payoff based on the Shapley value is
\begin{equation*}
    x_i(v) = \sum_{S \subseteq N \setminus \{ i\}} \frac{|S|!(n - |S| - 1)!}{n!}\left[v(S \cup \{i\}) - v(S)\right],
\end{equation*}
where $n$ is the total number of players in the game.

EMP is a ``fairer'' profit allocation mechanism proposed by \citet{FRISK2010}. The EPM mechanism was originally used to handle the cost allocation problem for collaborative forestry transportation planning. This allocation mechanism incorporates fairness by minimizing the relative profit gain among platforms, and it produces a stable allocation. However, the allocation does not guarantee to exist. With EPM, the relative profit of platform $i$ is defined as $\frac{x_i - v(\{i\})}{v(\{i\})} = \frac{x_i}{v(\{i\})} - 1$, indicating the ratio between profit gain after joining the alliance and foregoing profit. The difference in relative profit between any two platforms $i, j$ is $\frac{x_i}{v(\{i\})} - \frac{x_j}{v(\{j\})}$. The detailed Linear Program (LP) used for solving the core-guaranteed allocation is described in \ref{appen:cooperative}. 

Another important factor when considering the profit allocation between platforms is contribution. The profit distribution should consider the individual platform's contribution to the alliance. The contribution of each platform represents the proportion of the alliance's total profit related to customer and driver information provided by the platform. The contribution-based allocation mechanism is a core-guaranteed mechanism proposed by \citet{DAI2012633}. This profit allocation mechanism is first designed for collaborative logistics, where multiple carriers collaborate with each other to share transportation requests and vehicle capacity. It incorporates the fairness and stability of the allocation, but its existence is not ensured by this mechanism.
Formally, the contribution parameter for each platform $i$ is defined as
\begin{equation*}
    w_i = \frac{\theta_1 \cdot c_i + \theta_2 \cdot R_i}{\theta_1 \cdot \sum_{i \in N}c_i + \theta_2 \cdot \sum_{i \in N}R_i},
\end{equation*}
where $c_i$ and $R_i$ are the cost and revenue for platform $i$ in the alliance and $\theta_1$ and $\theta_2$ are two positive weights, which specifies the importance of offering supply and demand information to the profit creation of the alliance.
For each platform $i$, the cost indicates salaries for drivers who contract with platform $i$, and the revenue is represented by the payment of customers who request a ride via platform $i$.
$\theta_1$ and $\theta_2$ are two positive weights, which specify the importance of offering supply and demand information to the profit creation of the alliance. In this paper, we use the same definition for both parameters proposed by \citet{DAI2012633}.
The \textit{Profit Margin on Cost} parameter, $\theta_1$, is defined as the total net revenue of all platforms divided by the total cost of all platforms after collaboration. The \textit{Gross Profit Margin} parameter, $\theta_2$, is described by the total net revenue of all platforms divided by the total profit of all platforms. And the detailed LP to solve the contribution-based profit allocation is displayed in \ref{appen:cooperative}.

\subsection{Shared Mobility Marketplace Mechanism}

In the shared mobility marketplace, a central broker serves as an \emph{auctioneer}, collects and sells demand information to platforms according to certain mechanisms. Assuming that platforms bid for trip information based on their valuations, which are denoted by their net profits for serving a given trip. The central broker distributes customers to platforms to maximize the overall valuations and charges platforms the price of information. Under this market structure, platforms contract with drivers, set pricing scheme $\boldsymbol{o}^*$ for drivers, buy demand information from the central broker and match drivers with customers. The central broker collects customer requests, set pricing scheme $(\boldsymbol{p}^*, \boldsymbol{q}^*)$ for customers, and sells demand information to platforms. 

In this section, we propose one auction mechanism from the single-item auction. In the single-item auction, customer requests will be sold to platforms in sequential order. More complicated auction mechanisms (e.g., combinatorial auction) can be introduced in our framework in future studies. 

\subsubsection{Market Mechanism Design}

Consider the situation where a customer request is auctioned among $n$ platforms through a sealed-bid auction. In the sealed-bid auction, bidders place bids in sealed envelopes and simultaneously submit envelopes to the auctioneer, and the bidder with the highest price wins the item. For a customer request, each platform $i$ has a \textit{valuation} $v_i$ based on information of their drivers, indicating the net profit for serving it~\footnote{Platform has zero valuation if the trip can not be fulfilled or the net profit is negative.}. We would like to design an auction mechanism such that platforms bid truthfully (according to their valuations).

A naïve mechanism for the auctioneer is to assign the trip to the platform $i$ with the largest valuation, i.e., $i = \argmax_{j}v_j$. However, if the central broker receives untruthful biddings from platforms, implying that platforms bid prices other than their valuations, the trip will be assigned to the platform with the highest bid rather than the highest valuation. Therefore, we need to design an auction mechanism such that platforms bid truthfully. Meanwhile, it is necessary to charge the platform that wins the trip with the price of information. This necessary condition is proven by Proposition \ref{prop:1}.

\begin{prop}
\label{prop:1}
    If the auction mechanism assigns the trip for free to the platform with the highest bid, platforms will not bid truthfully according to their valuations.
\end{prop}
\begin{proof}
    Since platforms do not need to pay a price when increasing their bids, the dominant strategy for platforms is bidding as large as possible to win the trip if they have positive valuations. 
    Therefore, platforms will not bid truthfully.
\end{proof}

We introduce the VCG mechanism into the shared mobility market context to design a single-item auction mechanism. Under the VCG mechanism, bidders (or platforms) bid truthfully according to their true valuations \cite{Vickrey61}. The difference between the general auction and the auction in the shared mobility marketplace is that platforms have to serve customers to earn profit instead of owning them as assets. Considering a traditional auction situation where player $i$ wins the item with a valuation of 100 and payment of 98 in a general auction, it is reasonable for the winner to buy the item since the winner pays less money to get an item with a higher valuation. However, if the same case happens in the shared mobility marketplace auction, it is extremely unfair for platforms to participate in the auction because 98\% of their net profit goes to the central broker as the price of information.
To maintain a fair marketplace, we propose the following single-item auction mechanism:

\begin{prop}%[Single-item auction mechanism]
\label{prop:single_item}
    In the single-item auction in the shared mobility marketplace, the central broker sells customer requests in sequence to platforms with the highest bid, and the winning platform pays a price of information equal to $\gamma$ proportion of the second-highest bid, where $\gamma$ represents the rate for the price of information.
\end{prop}

In the shared mobility marketplace, the central broker's intention for collecting payments from platforms is to maintain a truthful-bidding auction instead of obtaining its own profit. Given such an auction mechanism, each platform will bid truthfully and the revenue of platforms can be protected.

In the implementation of the shared mobility marketplace mechanism, platforms and the central broker repeat the same auction procedure at the end of a given time interval $\Delta$, where $\mathcal{R} = \{r_1, ..., r_m\}$ represents a set of $m$ customer requests and $\mathcal{V}_1,...,\mathcal{V}_n$ indicate sets of available drivers for $n$ platforms. The central broker sells each customer request sequentially in a randomized order. When a customer request $r_j$ is offered to sell, the valuation $v_i$ by the platform $i$ equals the marginal revenue of bringing $r_j$ into the platform's request pool $R_i$. At the end of each time interval, platform $i$ will conduct a matching problem given the set of available drivers $V_i$ and the request pool $R_i$. When all platforms have non-positive valuations, the customer request $r_j$ will be sold again to platforms in the next time interval until the customer leaves the system (reach the maximum wait time).

\subsection{Central Trading Market Mechanism}
Platforms in the shared mobility market can not efficiently serve all customers because of the imbalance between supply and demand.
Under the central trading market mechanism, the central broker serves as the demand (customer request) information seller while the platforms are buyers.

Given a set of customer information $\mathcal{R}_i$ and a set of available drivers $V_i$ for each platform $i$ within a time period $\Delta$, platform $i$ conducts a matching step between customers and drivers and produces a set of unsatisfied requests $\Bar{\mathcal{R}}_i$ and a set of unmatched drivers $\Bar{\mathcal{V}}_i$. The central trading market works as follows: i) the central broker collects demand information $\Bar{\mathcal{R}}_i$, supply information $\Bar{\mathcal{V}}_i$, and pricing schemes $(\boldsymbol{p}^*, \boldsymbol{q}^*, \boldsymbol{o}^*)$ from each platform $i$; ii) the central broker distributes customer requests to platforms by solving a driver-customer matching problem between $\Bar{\mathcal{R}} = \{\Bar{\mathcal{R}}_1,...,\Bar{\mathcal{R}}_n \}$ and $\Bar{\mathcal{V}} = \{\Bar{\mathcal{V}}_1,...,\Bar{\mathcal{V}}_n \}$, which maximizes the total valuation; iii) each platform pays the price of information at rate $\gamma$ to other platforms who sell customer requests through the central broker.

If the customer request has not been traded, it will remain with its original platform until it is either served, traded, or left. When a driver from platform $i$ serves two requests from platforms $j, k$ in a shared ride and makes profit $p$ for the platform, platform $i$ will distribute the price of information $\gamma p$ to platforms $j$ and $k$ based on the profit of serving each request individually.

\subsection{Bilateral Trading Market Mechanism}

The bilateral trading market improves the system efficiency of $\emph{status quo}$ multi-platform market without introducing any third-party organization. In the bilateral trading market, platforms trade demand information directly with each other.
Similar to the central trading market, each platform $i$ has a set of unsatisfied requests $\Bar{\mathcal{R}}_i$ after conducting the matching step. However, unlike the central trading market where only available vehicles are considered, each platform can leverage all vehicles in the system when making a decision on whether to trade or not. Given $N = |\mathcal{P}|$ platforms in the market, there are ${N \choose 2}$ different bilateral trading possibilities within the shared mobility market. At each trading iteration, a random sequence of ${N \choose 2}$ trading possibilities is generated and platforms perform bilateral trading based on the generated sequence. For any two platforms $i, j$ with unsatisfied requests $\Bar{\mathcal{R}}_i, \Bar{\mathcal{R}}_j$, another random sequence is generated $\{ r \mid \forall r \in \Bar{\mathcal{R}}_i \cup \Bar{\mathcal{R}}_j \}$, indicating a series of demand information to be traded. A trade happens between the seller platform $i$ and the buyer platform $j$ if the platform $j$ can gain revenue $p$ from this trade when adding the request to the current request pool, indicating that there is a marginal benefit brought by serving an additional request from the other platform. The price of information $\gamma p$ is then transferred from platform $j$ to platform $i$. A traded request will not be included in future trading between platforms.

\section{Numerical Experiments}

In this section, we leverage a ride-sharing simulator using real-world ride-hailing data from NYC to evaluate the performances of proposed market structures. The simulators in this paper are models with Python 3.9.12 and solved with Gurobi 9.5.2~\cite{gurobi} on a 3.2 GHz Apple M1 Pro processor with 16 GB Memory.

\subsection{Simulation Overview} 

To evaluate the effectiveness of different market structures, we developed a simulation tool that uses real-world ride-hailing data from the Manhattan Borough of New York City. The simulation framework is shown in Figure \ref{fig:simulator}. The simulator models the trading and matching of customer requests and calculates various performance metrics, such as total vehicle miles traveled, percentage of unsatisfied requests, average customer wait time, and the number of operating trips, for each platform.

The simulation framework depicted in Figure 1 is a general framework that applies to all market structures, but variations are made for specific market structures. At each decision time interval, customer requests are collected, and the vehicle status is updated for each platform, including availability and location. Then, a trading stage is conducted following the trading mechanism for each proposed market structure. For the cooperative market, there is an additional module for profit allocation. However, for the status quo market with single or multiple platforms, the trading stage is not included. Once the trading stage is completed, customer requests are redistributed and returned, allowing platforms to optimize their matching of requests with available vehicles. This process is repeated at each time iteration.

\begin{figure}[!h]
    \centering
    \includegraphics[width=\textwidth]{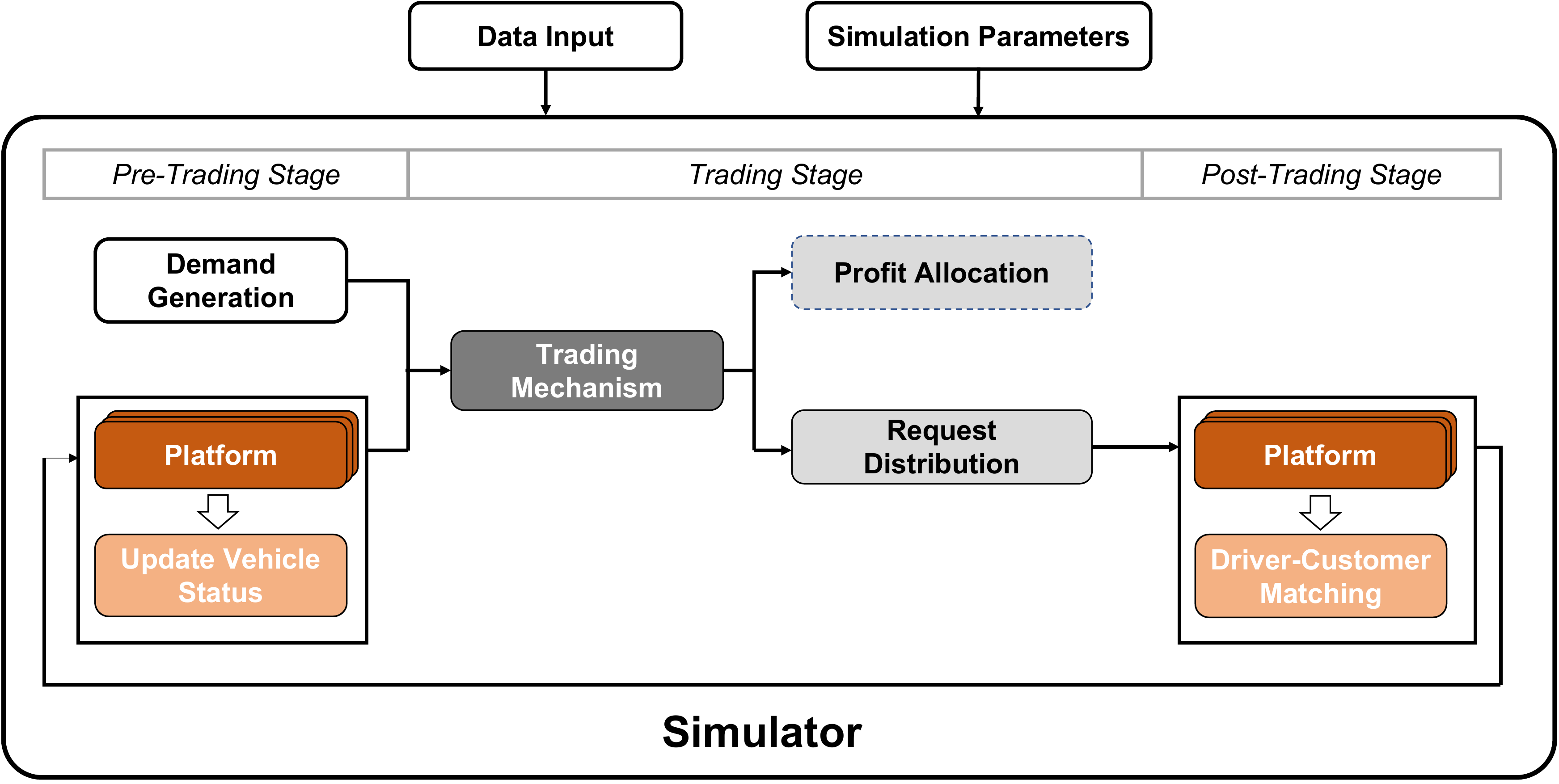}
    \caption{Simulation framework considering different market structures.}
    \label{fig:simulator}
\end{figure}

In our ride-sharing simulator, we simulated different market structures by considering either 2 or 3 platforms for each proposed market structure, and 1 to 3 platforms for the status quo market structure. The customer requests were randomly distributed among the platforms, with each platform receiving an equal number of requests. The number of drivers in the simulation were varied as 1200, 1800, 2400, and 3000. The drivers were also randomly assigned to each platform and each platform had an equal number of drivers. The initial locations of the drivers were randomly placed on the Manhattan road network.

\subsubsection{Data} 
The simulator uses various data sources. The ride-hailing demand data used in the simulation is from 7 to 10 a.m. on Wednesday, October 2, 2019 in Manhattan, obtained from a publicly available dataset that includes the specific time and regions (taxi zones) for both pick-up and drop-off for each customer request \cite{tlc_trip_record_data}. The detailed pick-up or drop-off locations were generated randomly from road nodes within each region. To calculate the travel distance for each trip, the road network for Manhattan was obtained from OpenStreetMap, which includes the length and permitted driving direction of each road segment. Additionally, travel time was estimated based on historical average traffic speed data provided by Uber Movement.

\subsubsection{Driver-Customer Matching} 
For the driver-customer matching component, we adapted the matching algorithm for on-demand ride-hailing systems proposed by \citet{Alonso-Mora2017}. This algorithm effectively formulates the trip-vehicle assignment as a constrained optimization problem that aims to minimize unsatisfied requests and promote car-sharing. Our simulation does not include the rebalancing stage for idle vehicles. The evaluation of trading requests in the shared mobility marketplace and bilateral trading market also uses the matching algorithm proposed by \citet{Alonso-Mora2017}. When an additional request is added, the matching problem is solved with the new request, and the marginal profit generated by the new request can be calculated.

\subsubsection{Simulation Parameters} 
The parameters and their values used in the ride-sharing simulator are shown in Table \ref{tab:parameter_definition}. The maximum detour factor for shared trips is 1.25, which means the shared trip is only feasible if the trip duration does not exceed 1.25 times the original trip duration. The maximum wait time for customers in the ride-hailing system is 300 seconds, and the maximum pickup time for customers when matching customers to drivers is 300 seconds. The penalty for each unsatisfied request in the matching problem is 10. The price of information rate when trading is 0.1 and the length of the time interval is 30 seconds.

The pricing scheme is based on Uber's pricing structure in NYC \cite{helling_2022,majaski_2022,driver_pay_rates_tlc}. The pricing differs between dedicated and shared trips. The trip price is calculated based on a base price plus a distance-based fare and a time-based fare, and a minimum fare is charged for each customer if the trip price is below the minimum price. It's worth noting that ride-hailing platforms use dynamic pricing mechanisms in real-world operations, but for the purpose of simplicity, a fixed pricing scheme is used in our simulations. Drivers get paid based on their travel distance and time. We assume that each platform $i \in \mathcal{P}$ in the simulator has an identical pricing scheme.

\begin{table}[h!]
    \centering
    \resizebox{0.9\textwidth}{!}{
    \begin{tabular}{ l | l | c }
     \hline 
     Parameter & Explanation & Value\\
     \hline \hline
     \multicolumn{3}{c}{\textbf{Simulator Environment}}\\
     \hline
     $\chi$ & Maximum detour factor for the shared ride & 1.25  \\
     $\Bar{\omega}_{wait}$ & Maximum wait time for customers & 300 (seconds)\\
     $\Bar{\omega}_{pickup}$ & Maximum pickup time for customers & 300 (seconds) \\
     $C$ & Penalty for each unsatisfied request & 10 \\
     $\gamma$ & Price of information rate & 0.1 \\
     $\Delta$ & Decision time interval length & 30 (seconds) \\
     \hline \hline
     \multicolumn{3}{c}{\textbf{Pricing Scheme for Platform $i \in \mathcal{P}$}}\\
     \hline 
     $\hat{p}_i$ & Base price for each dedicated trip & 2.55 (dollars)\\
     $p_i(d)$ & Dedicated trip price incurred by distance travelled & 1.75 (dollars/mile)\\
     $p_i(\tau)$ & Dedicated trip price incurred by time travelled & 0.35 (dollars/minute) \\
     $\barbelow{p}_i$ & Minimum fare for each dedicated trip & 8 (dollars) \\
     $\hat{q}_i$ & Base price for each shared trip & 1.22 (dollars)\\
     $q_i(d)$ & Shared trip price incurred by distance travelled & 0.81 (dollars/mile)\\
     $q_i(\tau)$ & Shared trip price incurred by time travelled & 0.26 (dollars/minute) \\
     $\barbelow{q}_i$ & Minimum fare for each shared trip & 7.84 (dollars)\\
     $o_i(d)$ & Driver earnings incurred by distance travelled & 1.429 (dollars/mile)\\
     $o_i(\tau)$ & Driver earnings incurred by time travelled & 0.502 (dollars/minute)\\
     \hline
     \end{tabular}}
    \caption{Simulation Parameters.}
    \label{tab:parameter_definition}
\end{table}

\subsection{Market Structure Evaluations}

\subsubsection{Overall Performance Comparisons}

To evaluate the performances of proposed market structures, we consider 11 different shared mobility markets with varying structures and platform numbers. These markets can be grouped into four categories: i) current market (status quo) with 1, 2, or 3 platforms, ii) bilateral and central trading markets with 2 or 3 platforms, iii) cooperative market with 2 or 3 platforms, and iv) shared mobility marketplace with 2 or 3 platforms. The most efficient market is represented by a single platform in the status quo, while multiple platforms in the status quo indicate the worst efficient market. The other three types of markets can decrease market inefficiencies and enhance system performance to some degree. The results for four scenarios with varying numbers of vehicles, 1200, 1800, 2400, and 3000, are presented. 

\begin{figure}[!h]
    \centering
    \includegraphics[width=\textwidth]{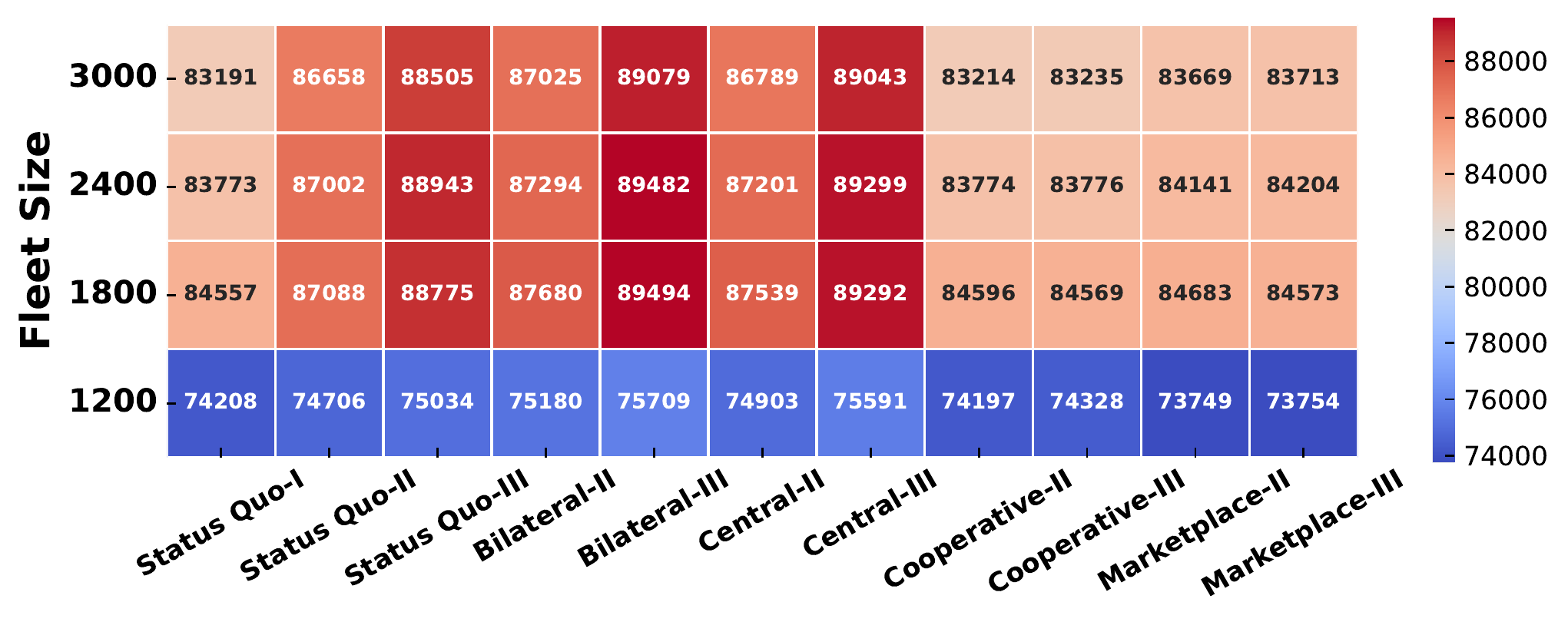}
    \caption{Total VMT for the complete shared mobility market.}
    \label{fig:total_VMT}
\end{figure}

The proposed market structures can reduce the VMT of the status quo competitive market up to 6\%. The first performance metric is the total VMT in the shared mobility market, as illustrated in Figure \ref{fig:total_VMT}. The results show that an increase in the number of vehicles leads to more total VMT as more customer requests can be satisfied, as shown in Figure \ref{fig:unsatisfied_request}. Given a scenario with a fixed fleet size, the status quo market with a single platform has the lowest total VMT, as well as the two cooperative markets. The cooperative market with three platforms and 3000 vehicles can reduce the VMT of the status quo competitive market by 6\% while satisfying 0.36\% more customers. The cooperative market achieves a similar level of efficiency as the monopolistic market while promoting healthy competition in the market. The shared mobility marketplace can also significantly decrease the total VMT from the status quo as the central broker assigns trips to platforms with the ``best'' vehicle to serve, up to 5.4\%. 

In contrast, the bilateral and central trading markets increase the total VMT for the shared mobility market. Both trading markets focus on unsatisfied requests, thus trading helps more customer requests to be served, resulting in more VMT in the system. The bilateral trading market has more total VMT than the central trading market as the bilateral market uses all vehicles and the central market only uses available (unmatched) vehicles. Therefore, more customers are satisfied with the bilateral trading market compared to the central trading market. Moreover, trading markets improve system efficiency more if the shared mobility market is more separated. 

\begin{figure}[!h]
    \centering
    \includegraphics[width=\textwidth]{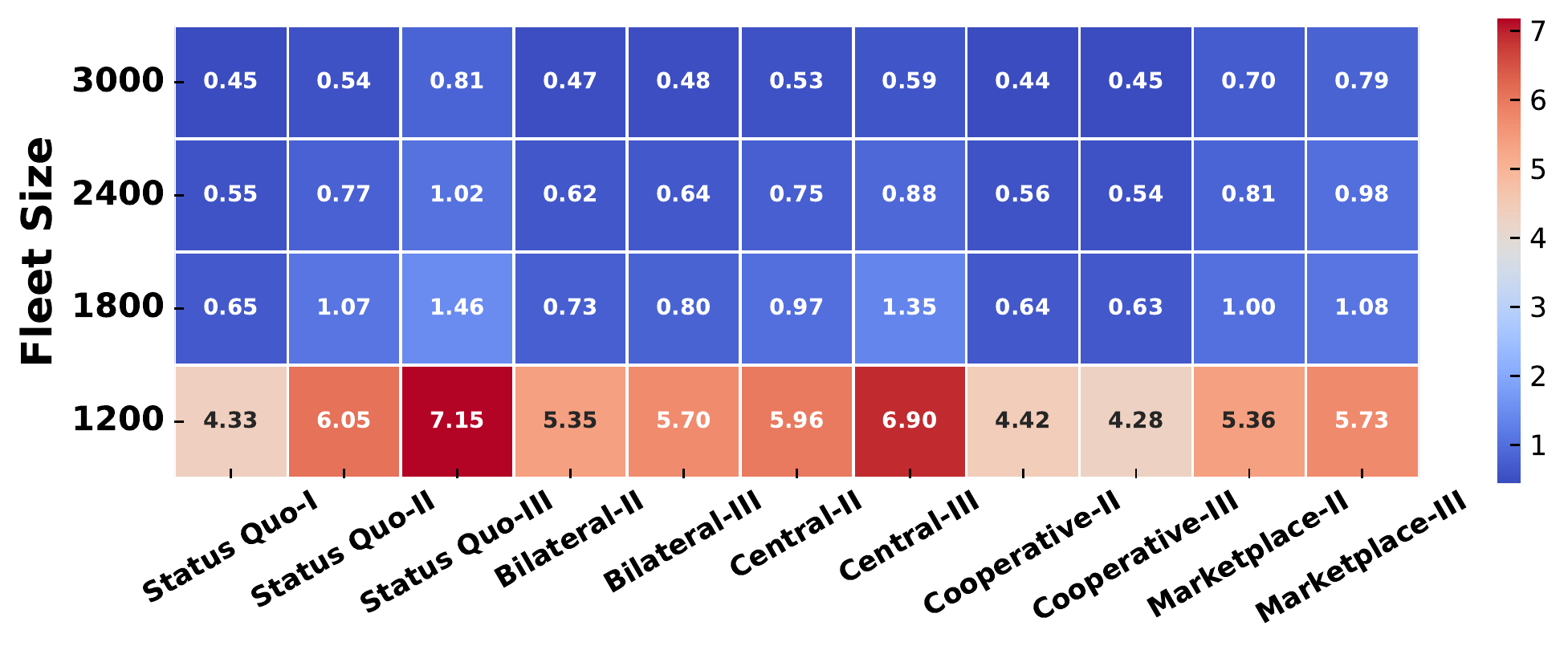}
    \caption{Percentage of unsatisfied requests for the complete shared mobility market.} 
    \label{fig:unsatisfied_request}
\end{figure}

More customers can be served by the proposed market structures with few exceptions. The second performance metric is the percentage of unsatisfied requests in the shared mobility market, as shown in Figure \ref{fig:unsatisfied_request}. It is evident that having more vehicles decreases the percentage of unsatisfied requests. In all scenarios, cooperative and trading markets are able to serve more customer requests than the divided market. The shared mobility marketplace's performance varies based on the number of vehicles and market segmentation. When there are fewer vehicles (1200 or 1800) or more vehicles (2400 or 3000) and a more segmented market (more than 3 platforms), the shared mobility marketplace is more effective at reducing the percentage of unsatisfied requests. However, when there are two platforms and enough vehicles (2400 or 3000), the shared mobility marketplace increases the percentage of unsatisfied requests. This is because the shared mobility marketplace prioritizes increasing platform profits and the standard matching process penalizes each unsatisfied request. Therefore, platforms may choose not to acquire unprofitable requests, which would have been fulfilled in the traditional market as there are no penalties for not fulfilling them.

The proposed market structures can improve the level of service for customers by reducing the average customer wait times, up to 5.4\%. The third performance metric is the average customer wait time in the shared mobility market, as illustrated in Figure \ref{fig:customer_wait_time}. The average customer wait time directly reflects the level of service for customers by a ride-sharing platform. As more vehicles are available in the market, the average customer wait time decreases. The cooperative market and the shared mobility marketplace both significantly reduce the average customer wait time, by up to 5.4\% and 3.2\%, respectively.

The shared mobility marketplace performs better than the monopolistic market when there are 1200 vehicles. This is because the central broker in the shared mobility marketplace assigns requests to platforms with the goal of maximizing profits. As a result, platforms may decline less profitable requests. This allows for vehicles to be used more efficiently in future iterations, resulting in shorter customer wait times. The benefits of declining unprofitable requests are even more pronounced in a more fragmented market. On the other hand, trading markets do not have a significant impact on the average customer wait time. The bilateral trading market has a higher likelihood of reducing wait times as all available vehicles are considered when fulfilling a traded request. However, the central trading market has a higher likelihood of increasing wait times, as requests that are not fulfilled are typically those with longer pickup times.

\begin{figure}[!h]
    \centering
    \includegraphics[width=\textwidth]{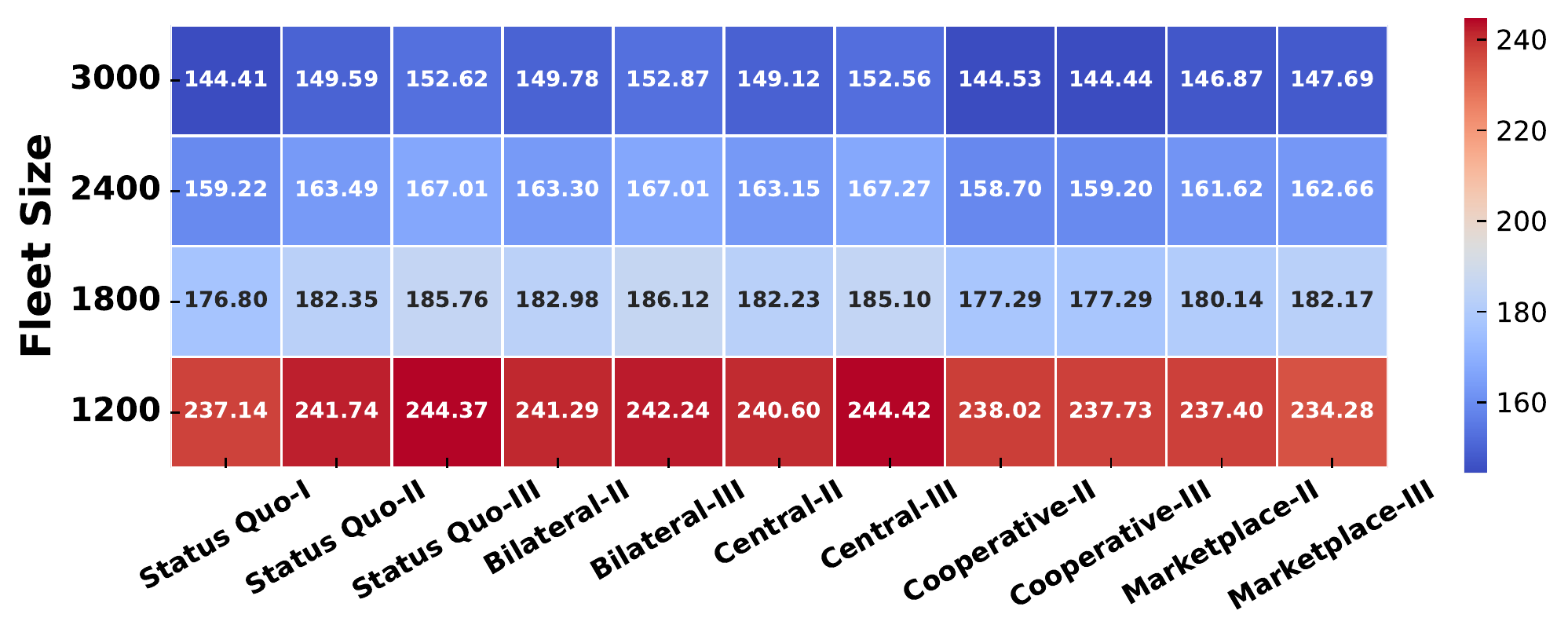}
    \caption{Average customer wait time for the complete shared mobility market.} 
    \label{fig:customer_wait_time}
\end{figure}

Customer demand can be served more efficiently with fewer trips for the proposed market structures, with up to 8.4\% fewer trips. The fourth performance measure is the total number of operating trips in a shared mobility market, as shown in Figure \ref{fig:operating_trips}. In a market that includes ride-sharing services, the number of trips taken is a good indicator of operational efficiency. An efficient ride-sharing system can meet more customer requests with fewer vehicles. The cooperative market is as efficient as a monopolistic market. The shared mobility marketplace also reduces the number of trips taken overall. In a scenario with 1800 vehicles, the number of trips taken is less than in a monopolistic market, as the shared mobility marketplace has more unsatisfied requests when they are less profitable. Both trading markets increase the number of trips taken, as more unsatisfied requests are fulfilled when platforms trade with each other.

\begin{figure}[!h]
    \centering
    \includegraphics[width=\textwidth]{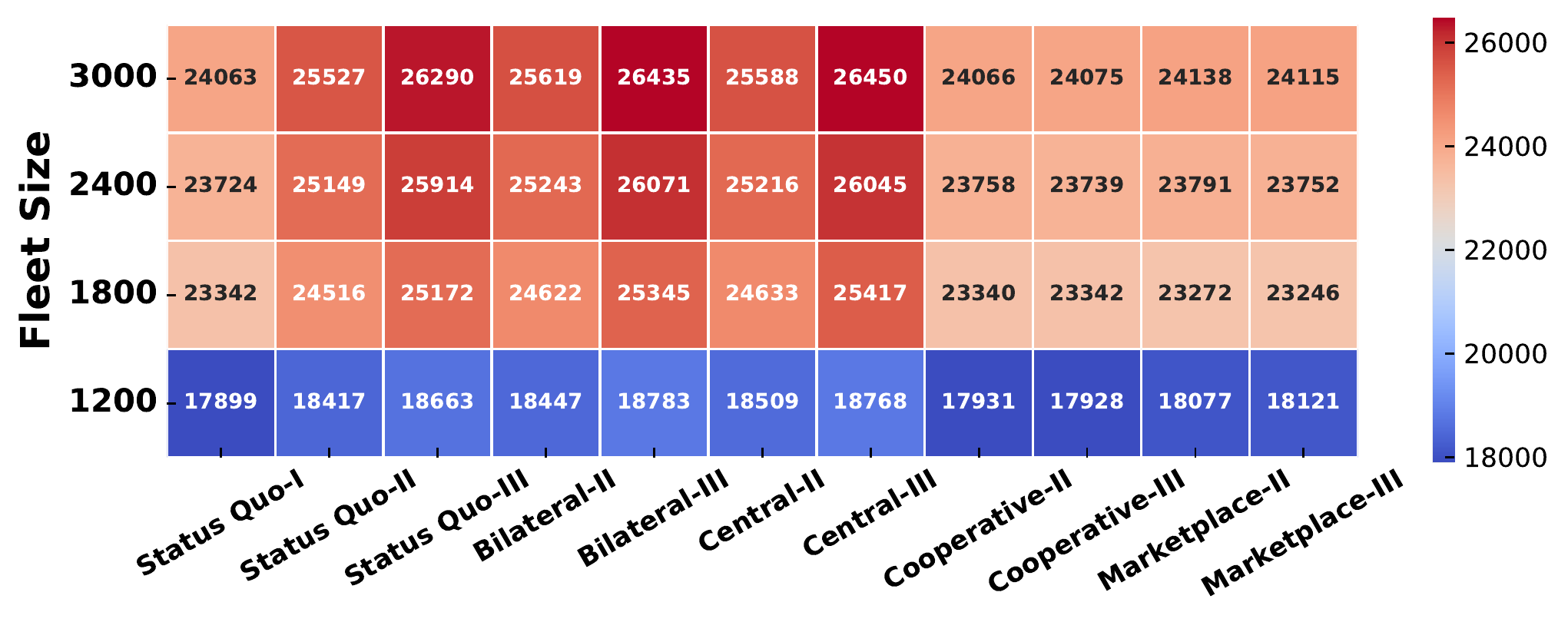}
    \caption{Total number of operating trips for the complete shared mobility market.} 
    \label{fig:operating_trips}
\end{figure}

Overall, the four market structures proposed can reduce market inefficiencies by better serving customer requests and providing benefits to platforms compared to the current segmented market. The cooperative market is the most efficient as all ride-sharing platforms work together. The shared mobility marketplace can greatly improve system performance, but may reject some less profitable customer requests. Trading markets require the least changes to the current shared mobility market. They can serve more customer requests. Meanwhile, bilateral trading markets provide additional benefits over central trading markets as platforms can leverage all available vehicles when evaluating requests from other platforms.

\subsubsection{More Evenly Distributed Profit}

The cooperative market requires a profit distribution mechanism that is fair for all platforms in the alliance. Figure \ref{fig:cooperative_market} shows the detailed analysis for the cooperative market with 2400 vehicles and 3 platforms. Profit distributions under three different distribution mechanisms are shown in Figure \ref{fig:cooperative_market_1}. Platform II receives a similar profit under all three distribution mechanisms, while Platforms I and III receive different profits. Platform III receives the most profit under the contribution-based distribution mechanism, and the least profit under the shapely value mechanism, while the opposite is true for Platform I.

For each platform, the amount of profit received should depend on the number of vehicles and the value of requests it contributes to the large alliance. Figure \ref{fig:cooperative_market_2} shows the number of contributing vehicles and the total value of contributing requests for each platform among all served trips. Platform III contributes the most vehicles, but requests with the least value. Since the values of contributed requests from the three platforms do not differ greatly, Platform III is allocated the most profit, which is further increased under the contribution-based distribution mechanism. As Platform I contributes the second highest value of requests, the EMP distribution mechanism distributes more profit to it compared to the contribution-based mechanism. The shapely value distribution mechanism is less equitable compared to the other two mechanisms as it lessens the impact of vehicle contributions for Platform III.

\begin{figure*}[!h]
    \centering
    \begin{subfigure}[b]{0.47\textwidth}
        \centering
        \includegraphics[width=\textwidth]{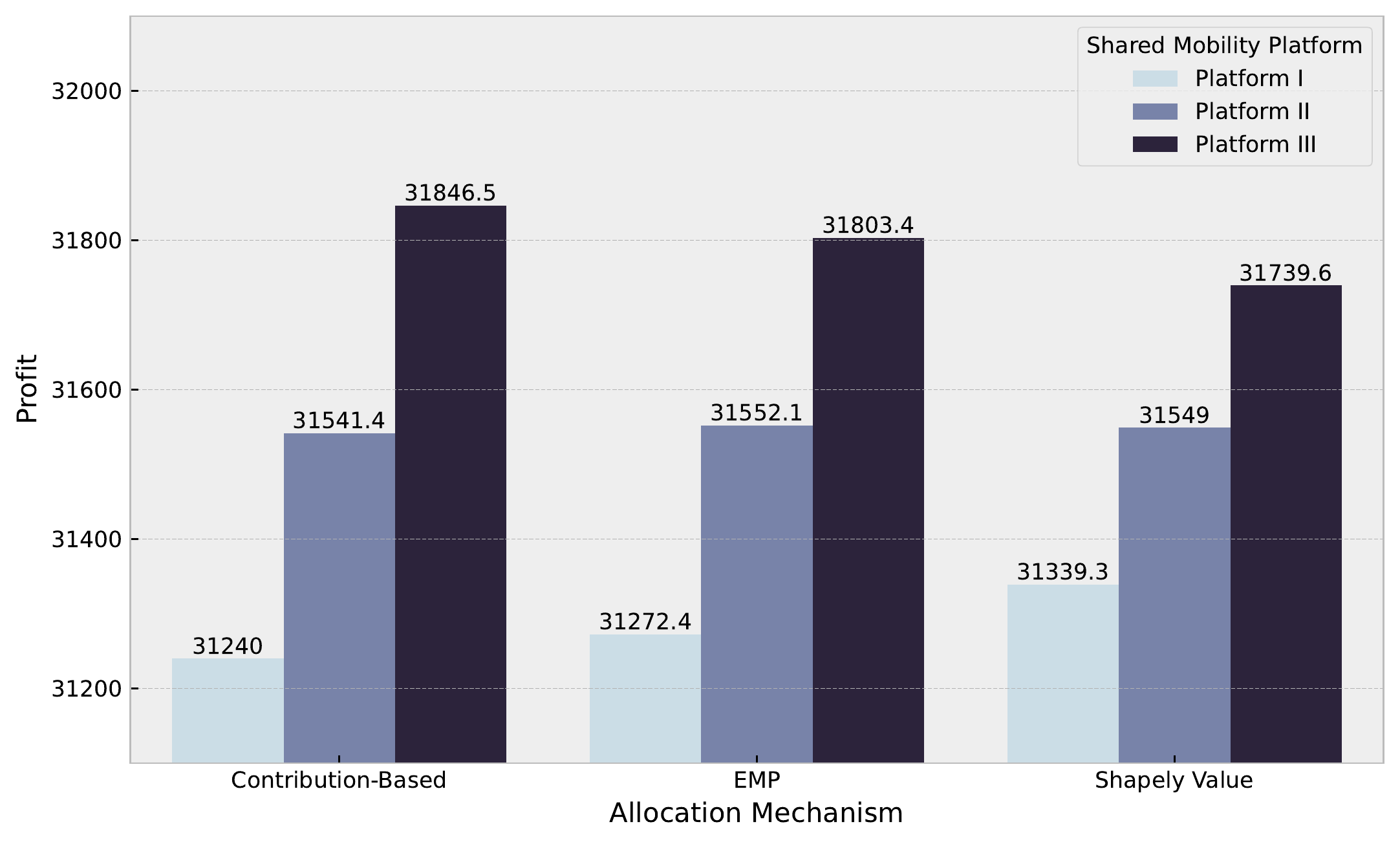}
        \caption[]%
        {Profit allocations for each profit allocation mechanism}
        \label{fig:cooperative_market_1}
    \end{subfigure}
    \hfill
    \begin{subfigure}[b]{0.51\textwidth}  
        \centering 
        \includegraphics[width=\textwidth]{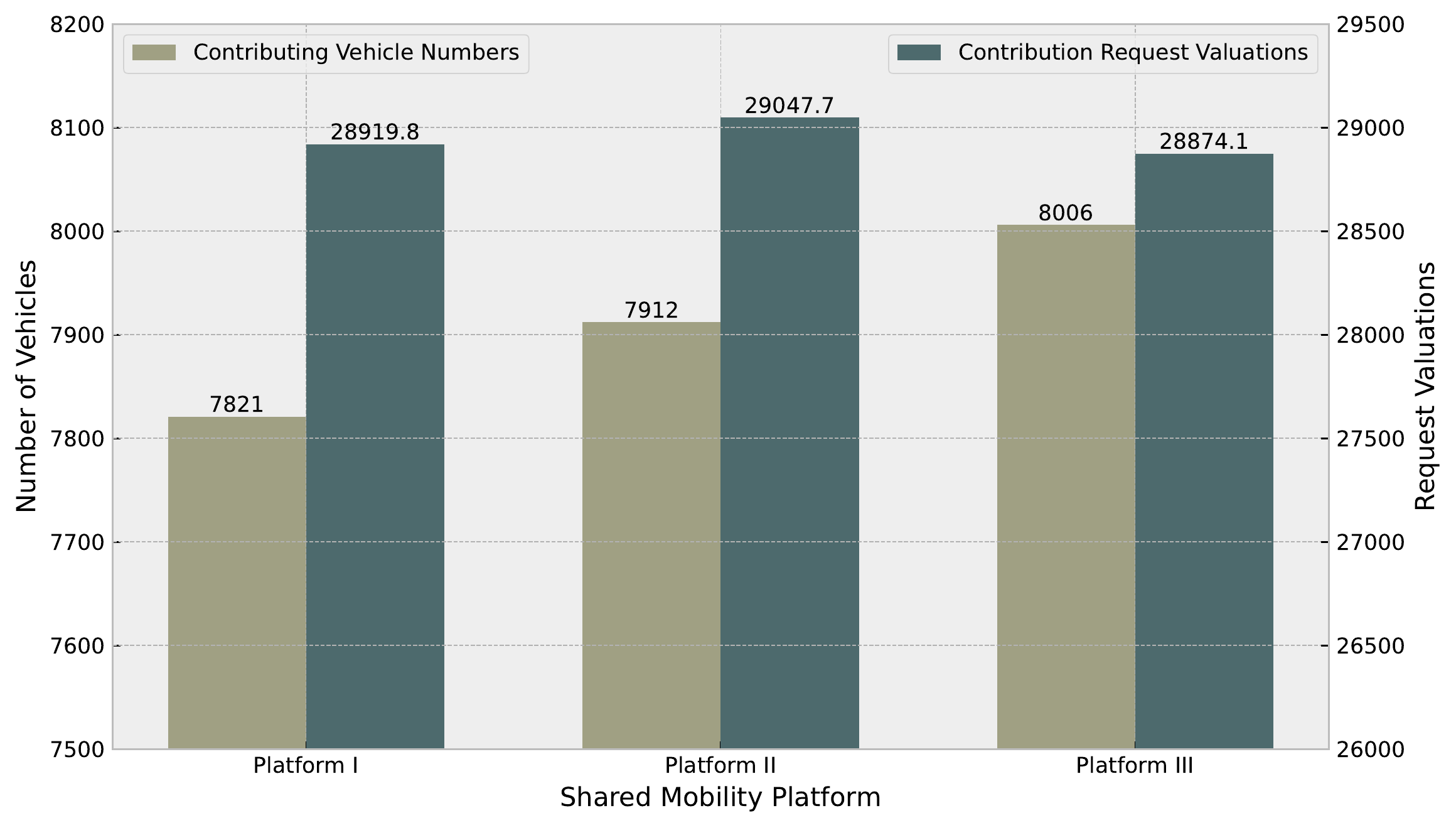}
        \caption[]%
        {Contributing vehicles and requests for each platform}    
        \label{fig:cooperative_market_2}
    \end{subfigure}
    \caption{Detailed analyses for the cooperative market with 2400 vehicles and 3 platforms.} 
    \label{fig:cooperative_market}
\end{figure*}

\subsubsection{Best Vehicles Win Customer Requests in Auction}

In the shared mobility marketplace, a central broker distributes customer requests to platforms through an auction process. Platforms bid for requests based on their perceived value and the winner pays the second-highest bid multiplied by a rate known as the price of information ($\gamma$). Figure \ref{fig:shared_mobility_marketplace} provides a detailed analysis of this marketplace structure with 2400 vehicles and 3 platforms.

\begin{figure*}[!h]
    \centering
    \begin{subfigure}[b]{0.47\textwidth}
        \centering
        \includegraphics[width=\textwidth]{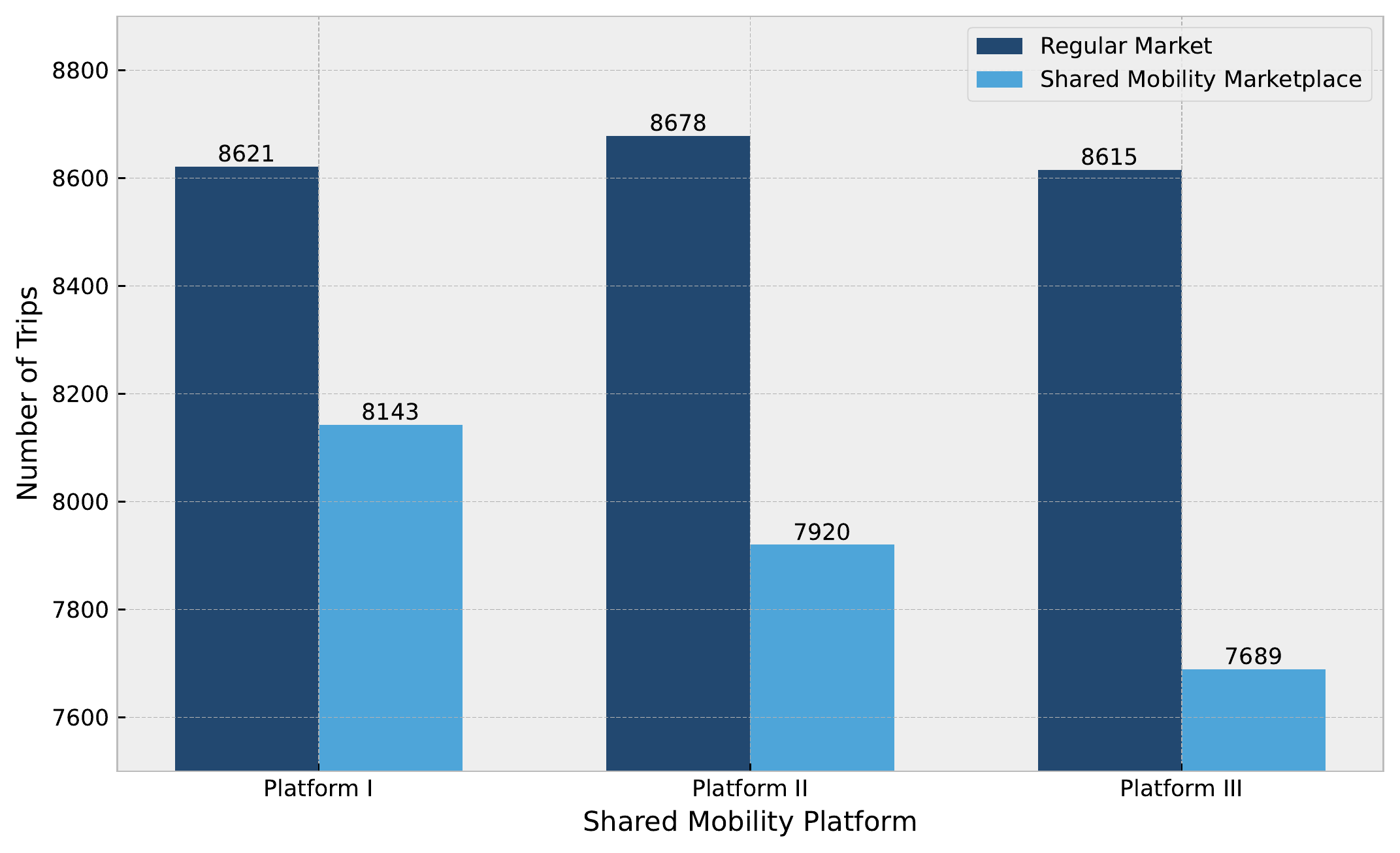}
        \caption[]%
        {Number of operating trips between status quo market and shared mobility marketplace for each platform}
        \label{fig:shared_mobility_marketplace_1}
    \end{subfigure}
    \hfill
    \begin{subfigure}[b]{0.51\textwidth}  
        \centering 
        \includegraphics[width=\textwidth]{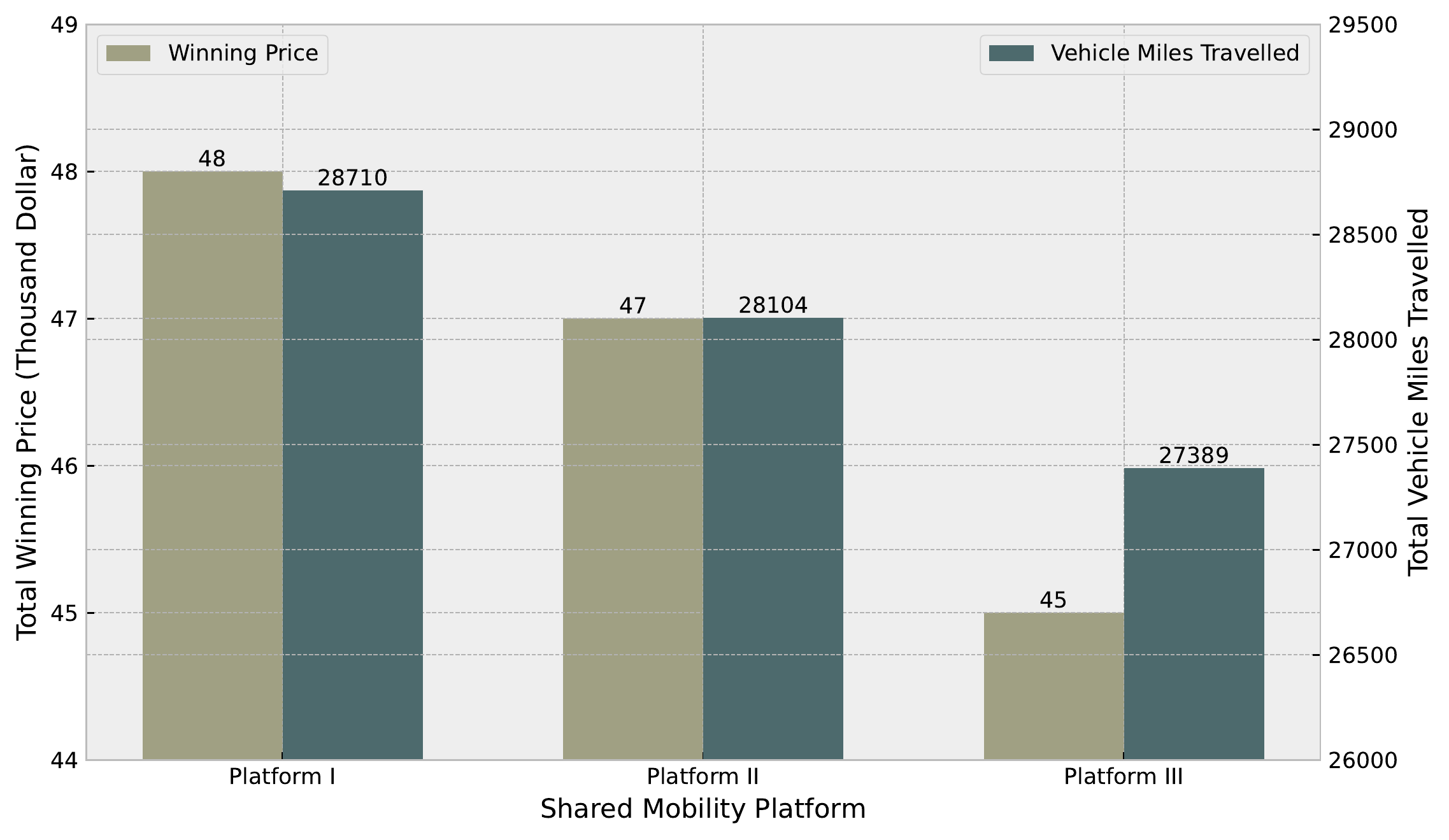}
        \caption[]%
        {Total winning prices and total VMT for each platform in the shared mobility marketplace}    
        \label{fig:shared_mobility_marketplace_2}
    \end{subfigure}
    \caption{Detailed analyses for the shared mobility marketplace with 2400 vehicles and 3 platforms.} 
    \label{fig:shared_mobility_marketplace}
\end{figure*}

The shared mobility marketplace can significantly reduce the total number of trips needed to serve all customers, and the platform with the best vehicle locations wins customer requests. Compared to the traditional segmented market, the shared mobility marketplace can significantly reduce the number of trips taken by each platform, as shown in Figure \ref{fig:shared_mobility_marketplace_1}. This is due to better utilization of vehicles and the rejection of less profitable requests. Figure \ref{fig:shared_mobility_marketplace_2} illustrates the total winning price and total VMT for each platform. Platform I takes the most trips as it pays the highest winning price and has the highest total VMT. It's worth noting that the shared mobility marketplace distributes customer requests to the platform that has the best vehicle to serve them, which can lead to uneven allocation of requests among platforms.

\subsubsection{More Tradings when Having More Demand}

In the bilateral and central trading markets, unsatisfied requests are exchanged between platforms after each matching step. As seen in Figure \ref{fig:trading_markets}, the number of trades in each iteration and the demand level, as represented by a 10-iteration rolling average, are displayed for both markets.

\begin{figure*}[!h]
    \centering
    \includegraphics[width=\textwidth]{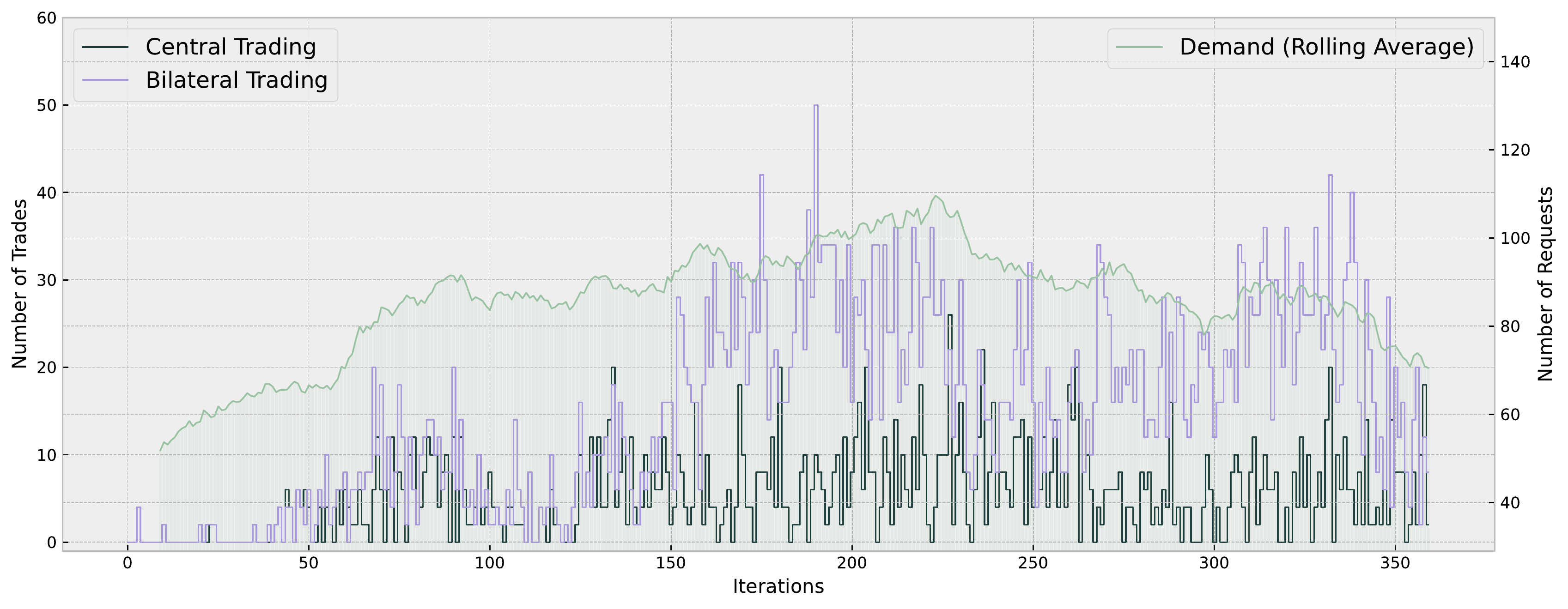}
    \caption[]{Number of trades and demand levels (10-iteration rolling average) for central and bilateral trading markets with 2400 vehicles and 3 platforms.}
    \label{fig:trading_markets}
\end{figure*}

There are more trading opportunities when having more demand. 
Initially, fewer trades occur as platforms have an abundance of vehicles. However, as demand increases, the number of trades increases. The bilateral trading market sees more trades between platforms than the central market. This is because the bilateral market makes use of all available vehicles when assessing requests from other platforms, allowing for the possibility of adding the request to an existing trip or matching it to a vehicle that is already matched to another request. This evaluation process in the bilateral market can greatly enhance the potential benefits of trading customer requests. On the other hand, the central trading market can only utilize unmatched vehicles in the trading process, leading to a waste of vehicle resources and fewer trading possibilities. 

\section{Summary and Discussion}

This paper aims to address the segmentation in shared mobility markets with multiple platforms. While a monopolistic market may have the highest system efficiency, it lacks healthy competition among platforms. To improve the efficiency of ride-sharing platforms while preserving competition, we propose four market structures utilizing a unified framework. Each proposed structure is analyzed and its market mechanisms are discussed. We use a ride-sharing simulator, incorporating real-world ride-hailing data from NYC, to evaluate the performance of the proposed market structures. With the proposed market structure, the total VMT can be reduced by up to 6\% while serving more customers. The customer wait time can be reduced by up to 5.4\% and all customers can be served by utilizing 8.4\% fewer trips.

Although the proposed market structures can bring benefits to the status quo shared mobility market, achieving them might induce some feasibility discussions. 

\textbf{Cooperative Market}: The cooperative market exists in numerous industries, including forestry transportation and logistics for instance~\cite{FRISK2010, DAI2012633}. The most critical component is a fair profit allocation mechanism. Large platforms might hesitate to join the alliance as their market shares could decline after the cooperation. The small-scale platforms have larger gains from the cooperation compared to large platforms. Meanwhile, a unified pricing scheme is required for multiple platforms, which could lead to a cartel that sets a higher price to gain monopolist-like benefits. Customers' interests can be hurt in the cooperative market. Therefore, government regulation is necessary for cooperative markets.

\textbf{Shared Mobility Marketplace}: There are several transportation-related markets implementing auction mechanisms.  In the truckload transportation auction market, retailers, manufacturers, distributors, and other companies which need to move freight are auctioneers, and the trucking companies that own the transportation assets serve as bidders~\cite{Cramton2006}. For the bus routes market in the Greater London area, the London Regional Transport (LRT), which is succeeded by Transport for London (TfL), acts as an auctioneer and sells rights for carrying out bus services to private operators~\cite{Cramton2006}. These instances suggest possibilities for introducing auction mechanisms in the shared mobility market.

For the shared mobility marketplace, platforms are significantly affected due to the loss of control for demand information. Large platforms lose advantages compared to small platforms, and they are unlikely to join the shared mobility marketplace unless receiving external interventions. 
Meanwhile, the central broker in this market requires technological competencies for gathering enormous demand information and distributing demand information rapidly to maintain a satisfying service for customers, which is required by the on-demand nature of shared mobility services. The central broker needs to be regulated because of its power of setting the pricing scheme and collecting all customers' information. 

\textbf{Central Trading Market}: The central trading market is a common type of market in the real world, such as stock and rental markets. The key feature of a central trading market is that it enables all participants to remain anonymous, but the only relevant factor is the price. In a stock exchange, for example, buyers and sellers do not care about who they are buying from or selling to, as long as the price meets their requirements. Both large and small companies benefit from this market without worrying about a decline in their market share. For large platforms, owning more customer or driver information enhances their earning potential by selling it to other platforms. A central broker in this market needs technical competencies to maintain a trading platform for exchanging demand information with platforms.

\textbf{Bilateral Trading Market}: In most real-world markets, there are multiple participants, which can make bilateral trading less efficient than central trading due to the lack of information sharing. However, the bilateral trading market requires minimal intervention in the current shared mobility market. In this market, platforms increase their profits by trading request information. Similar to the central trading market, large platforms do not have to worry about losing market share and can benefit from trading information.

In this paper, market efficiency is measured using metrics such as total VMT, the number of operating trips, the percentage of unsatisfied requests, and average customer wait times. It is acknowledged that efficiency can be evaluated using a more comprehensive set of objectives such as total travel time, total carbon emissions, total energy consumption, etc, which are worthy of further research. Additionally, the equity of platforms, customers, and drivers in shared mobility markets, and the redistribution of surplus from improved market efficiency, merit further discussion in future research.

Furthermore, the impact of multi-homing drivers and customers, who form a significant portion of the shared mobility market, has not been considered in the design of different markets. Their role in improving the efficiency of segmented shared mobility markets should be more accurately assessed.

Lastly, simple market mechanisms have been proposed for each market structure. More complex mechanisms, such as a combinatorial auction mechanism, can also be proposed for the shared mobility marketplace.

In conclusion, market design in the shared mobility domain can be a powerful tool to reduce inefficiency caused by the segmentation of different platforms and to promote healthy competition through collaboration. We hope that this paper serves as a starting point for further research to quantify efficiency improvements and provide more detailed mechanism designs. Our work is able to provide insights for TNC operators, transportation authorities, transportation engineers, and urban planners.

\section*{Acknowledgements}
\noindent This research project is supported by the Singapore--MIT Alliance for Research and Technology Centre (SMART), and by the Institute of Public Governance, Peking University (General Project YBXM202114). The authors thank Nicholas S. Caros and Shenhao Wang for their comments on the paper. Data retrieved from Uber Movement, (c) 2023 Uber Technologies, Inc., https://movement.uber.com.

\section*{Author Contribution Statement}
The authors confirm their contribution to the paper as follows: study conception and design: X. Guo, H. Zhang, P. Noursalehi, J. Zhao; data collection: X. Guo, A. Qu, H. Zhang; analysis and interpretation of results: X. Guo, A. Qu, H. Zhang, J. Zhao; draft manuscript preparation: X. Guo, A. Qu, H. Zhang, J. Zhao. All of the authors reviewed the results and approved the final version of the manuscript. The authors do not have any conflicts of interest to declare.

\bibliographystyle{model1-num-names}
\bibliography{reference.bib}

\newpage
\appendix

% \section{Glossary}
% \begin{table}[!h]
% \small
% \centering
% \caption{Notations used in this paper}
% \label{tab:glossary}
% \begin{tabular}{ll}
% \hline
% \underline{Unified theoretical framework}\\
% $G=(V,E)$ & Shareability network with vertex set \\
% \underline{Market description}\\
% $G=(V,E)$ & Shareability network with vertex set \\
% \underline{Market mechanisms}\\
% $N$ & Set of players (platforms) \\
% $v(\cdot)$ & Valuation (net profit) function \\ 
% \hline
% \end{tabular}
% \end{table}

\section{Detailed Mechanisms for Cooperative Market}
\label{appen:cooperative}
% \subsubsection{Shapley Value}

% It is built on four axioms: \textit{Efficiency}, \textit{Symmetry}, \textit{Null player} and \textit{Additivity}. 
% \textit{Efficiency} axiom means the allocation is efficient.
% \textit{Symmetry} axiom indicates that players who are treated identically by the characteristic function will be assigned with the same value in the allocation.
% \textit{Null player} axiom shows the allocation value for a null player $i$, who satisfies $v(S \cup \{i\}) = v(S)$ for all coalitions $S$ expect $i$, is zero.
% \textit{Additivity} axiom implies that when combining two cooperative games with characteristic functions $v, w$, the allocation value for each player $i$ is $x_i(v) + x_i(w)$, where $x_i(v)$ and $x_i(w)$ are allocation values in two games.

\subsection{Equal Profit Method}

% One shortcoming of the Shapley value is that it is hard for participants to have a fair profit gain after joining the alliance, which is an important incentive for platforms to cooperate with each other.

% In order to guarantee the allocation is in the core, we introduce two additional constraints in Definition \ref{def:core}.
The profit allocation of EPM is solved by the following LP:

\begin{subequations}
\begin{align}
\min \quad
& \alpha \\
\text{s.t.} \quad
& \alpha \geq \frac{x_i}{v(\{i\})} - \frac{x_j}{v(\{j\})} \quad \forall i,j \in N\\
& \sum_{i \in S} x_i \geq v(S) \quad \forall S \subset N \\
& \sum_{i \in N} x_i = v(N) \\
& x_i \geq 0 \quad \forall i \in N
\end{align}
\end{subequations}
\\
The objective function (1a) minimizes the largest difference in relative profit between any two platforms in the grand coalition $N$.
Constraints (1c) and (1d) guarantee that the profit allocation is at the core of the cooperative game $(N, v)$.
Constraints (1e) ensure the non-negativity of the profit allocation for platforms.

The EPM allocation mechanism ensures the final profit allocation is in the core and it is equitable for platforms regarding relative profit gain after joining the alliance.
However, problems (1a) - (1e) can be infeasible since the core of the cooperative game can be empty.

\subsection{Contribution-Based Allocation Mechanism}

The profit allocation of the contribution-based allocation mechanism is solved by the following LP:

\begin{subequations}
\begin{align}
\min \quad
& \beta \\
\text{s.t.} \quad
& \beta \geq \frac{x_i - v(\{i\})}{w_i} - \frac{x_j - v(\{j\})}{w_j} \quad \forall i,j \in N\\
& \sum_{i \in S} x_i \geq v(S) \quad \forall S \subset N \\
& \sum_{i \in N} x_i = v(N) \\
& x_i \geq 0 \quad \forall i \in N
\end{align}
\end{subequations}
\\
The objective function (2a) finds a profit allocation that minimizes the difference in profit-contribution ratio between any two platforms.
Constraints (2c) and (2d) assure that the profit allocation is at the core.
Constraints (2e) ensure the non-negativity property.

The contribution-based allocation mechanism is guaranteed to be in the core of the cooperative game $(N, v)$, but it does not guarantee the existence of the allocation since problems (2a) - (2e) can be infeasible.

\end{document}